\documentclass[preprint,a4paper,10pt]{article}
\title{On a definition of multi-Koszul algebras}
\author{Estanislao Herscovich
\footnote{Institut Joseph Fourier, Universit\'e Grenoble I, France.}, 
Andrea Rey \footnote{Departamento de Matem\'atica y Ciencias, Universidad de San Andr\'es, Argentina.}}
\date{}

\usepackage{amsmath,amsthm,amsfonts,amssymb,stmaryrd}
\usepackage{latexsym}
\usepackage{fullpage}
\usepackage{a4,color,palatino,fancyhdr}
\usepackage{graphicx}%poner [dvips] para gráficos
\usepackage{float}  %Este packete permite controlar la flotacion de figuras al poner [H] o [h] como opcion de \begin{figure}
\usepackage[small, bf, margin=90pt, tableposition=bottom]{caption}
\RequirePackage{amssymb}
\RequirePackage[T1]{fontenc}
\usepackage[breaklinks, naturalnames]{hyperref}
\usepackage{amsrefs}

\input{xy}
\xyoption{all}
\xyoption{poly}
\usepackage[all]{xy}

\newtheorem{theorem}{Theorem}[section]
\newtheorem{proposition}[theorem]{Proposition}
\newtheorem{definition}[theorem]{Definition}
\newtheorem{lemma}[theorem]{Lemma}
\newtheorem{corollary}[theorem]{Corollary}
\newtheorem{remark}[theorem]{Remark}
\newtheorem{example}[theorem]{Example}

\numberwithin{equation}{section}

\def\cl#1{{\langle #1 \rangle}}

\def\place{{-}}
\def\pd{\mathop{\rm pdim}\nolimits}
\def\Ker{\mathop{\rm Ker}\nolimits}
\def\im{\mathop{\rm Im}\nolimits}

\def\Tor{\mathop{\rm Tor}\nolimits}
\def\Ext{\mathop{\rm Ext}\nolimits}
\def\ext{\mathop{\rm ext}\nolimits}
\def\Hom{{\mathrm {Hom}}}

\def\grMod{{\mathrm {grMod}}}

\newcommand\ZZ{{\mathbb{Z}}}
\newcommand\NN{{\mathbb{N}}}

\def\place{{-}}

\begin{document}

\maketitle
                                                     
\hrulefill
%%%%%%%%%%%%%%%%%%%%%%%%%%%%%%%%%%%%%%%%%%%%%%%%%%%%%%%%%%%%%%%%%%%%%%%%%%%%%%%%%%%%%%%%%%%%%%%%%%%%%%%%%%%%%%%%%%%%%%%%%%%%%%%%%%

\begin{abstract}
In this article we introduce the notion of \emph{multi-Koszul algebra} for the case of a nonnegatively graded connected algebra with a finite number of generators of degree $1$ and with a finite number of relations, as a generalization of the notion of (generalized) Koszul algebras defined by R. Berger for homogeneous algebras, 
which were in turn an extension of Koszul algebras introduced by S. Priddy. 
Our definition is in some sense as closest as possible to the one given in the homogeneous case. 
Indeed, we give an equivalent description of the new definition in terms of the \textrm{Tor} (or \textrm{Ext}) groups, similar to the existing one for homogeneous algebras, 
and also a complete characterization of the multi-Koszul property, which derives from the study of some associated homogeneous algebras, 
providing a very strong link between the new definition and the generalized Koszul property for the associated homogeneous algebras mentioned before. 
We further obtain an explicit description of the Yoneda algebra of a multi-Koszul algebra. 
As a consequence, we get that the Yoneda algebra of a multi-Koszul algebra is generated in degrees $1$ and $2$, so a $\mathcal{K}_{2}$ algebra in the sense of T. Cassidy and B. Shelton. 
We also exhibit several examples and we provide a minimal graded projective resolution of the algebra $A$ considered as an $A$-bimodule, 
which may be used to compute the Hochschild (co)homology groups. 
Finally, we find necessary and sufficient conditions on some (fixed) sequences of vector subspaces of the tensor powers of the base space $V$ to obtain in this case the multi-Koszul property in the case we have relations in only two degrees. 
\end{abstract}



\hrulefill
%%%%%%%%%%%%%%%%%%%%%%%%%%%%%%%%%%%%%%%%%%%%%%%%%%%%%%%%%%%%%%%%%%%%%%%%%%%%%%%%%%%%%%%%%%%%%%%%%%%%%%%%%%%%%%%%%%%%%%%%%%%%%%%%%%%%%%%%%%%%%%%%%%%%%%%%

\tableofcontents

\section{\texorpdfstring{Introduction}{Introduction}}
\label{sec:int}

Koszul algebras were introduced by S. Priddy in \cite{P}, motivated by the article \cite{K} published by J.-L. Koszul in the 50s. 
They have been extensively studied in the last years, in particular due to their applications in representation theory (\textit{cf.} \cite{BGS1}, \cite{BGS2}), algebraic geometry (\textit{cf.} \cite{F}), quantum groups (\textit{cf.} \cite{M}), and combinatorics (\textit{cf.} \cite{HL}), to mention a few. 
These algebras are necessarily quadratic, \textit{i.e.} they are of the form $T(V)/\cl{R}$, with $R \subseteq V^{\otimes 2}$. 
R. Berger generalized in \cite{B2} the notion of Koszul algebras (\textit{cf.} also \cite{GMMZ}) to the case of homogeneous algebras, 
\textit{i.e.} algebras given by $T(V)/\cl{R}$, with $R \subseteq V^{\otimes N}$, for $N \geq 2$. 
They were called \emph{generalized Koszul}, or \emph{$N$-Koszul} if the mention to the degree of the relations was to be indicated, and the case $N = 2$ of the definition introduced by Berger 
coincides with the one given by Priddy. 
The general definition shares a lot of good properties with the one given by Priddy, justifying the terminology (see for example \cites{B2, BG}). 
In particular, the Yoneda algebra of an $N$-Koszul algebra is finitely generated (in degrees $1$ and $2$), and its structure is easily computed from that of the original algebra. 
We would like to point out that the new class of algebras satisfying the Koszul property of Berger lacks however of other interesting properties, 
\textit{e.g.} they are not closed under taking duals, or under considering graded Ore extensions, the Yoneda algebra of an $N$-Koszul algebra is not formal for $N \geq 3$, etc. 

On the other hand, several examples of not necessarily homogeneous algebras which arise in the practice and which share some of the interesting properties of generalized Koszul algebras lead to the question if there is an analogous definition of \emph{Koszul-like} algebra for more general situations. 
In this article we propose such a definition for the case of a finitely generated nonnegatively graded connected algebra which is generated in degree $1$ and has a finite number of relations, \textit{i.e.} algebras of the form $T(V)/\cl{R}$, where $V$ is a finite dimensional vector space, which we consider to be in degree $1$, and $R \subseteq T(V)_{\geq 2}$ is a finite dimensional graded vector space. 
They will be called \emph{multi-Koszul}. 
Our main goal is to provide such a class of algebras, which are in some sense the closest possible to the generalized Koszul algebras, 
for which the Yoneda algebra is in fact finitely generated and its structure is directly deduced from that of the original algebra. 

The new definition may seem however to be too restrictive (\textit{e.g.} see Remark \ref{rem:malo}, and Corollary \ref{coro:monomial}), 
and despite the fact that it is probably not the most general possible and reasonable extension of the Koszul property for such algebras, 
all the nice properties satisfied by it (\textit{e.g.} Theorem \ref{thm:koszulbypeaces}, Propositions \ref{prop:libre}, \ref{prop:tresimp} and \ref{prop:yoneda}, Remarks \ref{rem:yoneda} and \ref{rem:yonedaainf}, 
and Corollary \ref{coro:yoneda}) 
make us believe that any sensible such general definition of \textit{Koszul-like} algebra in the general context of graded algebras, if it exists, 
should necessarily include our definition as a special case. 

This work is partially inspired on the Ph.D. thesis of the second author, but it considers a more general setting of graded algebras.

The contents of the article are as follows. 
We start by recalling in Section 2 several well-known definitions and results about the category of graded modules over a nonnegatively graded connected algebra.

Section 3 is devoted to the definition of multi-Koszul algebras and to prove some properties for this family of algebras. 
The first main result, Proposition \ref{prop:tresimp} (see also Proposition \ref{prop:tresimpext}) gives a (co)homological description of multi-Koszul algebras in term of their \textrm{Tor} 
(or \cal{E}xt) groups, which yields a left-right symmetry of the definition.
Then, Theorem \ref{thm:koszulbypeaces} relates the multi-Koszul property for an algebra $A$ to the (generalized) Koszul property of some homogeneous algebras associated to $A$. 
Moreover, we also study the associated Yoneda algebra and prove in Corollary \ref{coro:yoneda} that multi-Koszul algebras are $\mathcal{K}_2$ algebras, 
in the sense defined by B. Cassidy and T. Shelton in \cite{CS}. 
We further obtain a description of the $A_{\infty}$-algebra structure of the corresponding Yoneda algebra (see Remarks \ref{rem:yoneda} and \ref{rem:yonedaainf}). 

Following the ideas of Berger in \cite{B2}, in Section 4 we construct an $A$-bimodule resolution of a multi-Koszul algebra, 
which can be used in the computation of the Hochschild (co)homology groups of the algebra.

Finally, in Section 5 we concentrate on the special case of having relations in only two degrees. 
We find necessary and sufficient conditions on some (fixed) sequences of vector subspaces of the tensor powers of the base space $V$ to get the multi-Koszul property, 
following the lines of the analysis done in the case of (generalized) Koszul algebras by Berger (see \cite{B2}, Section 2, but \textit{cf.} also \cite{BF}). 
Some of these subspaces together with their conditions are the ones already found in the homogeneous case, but these are not equivalent to the multi-Koszul definition, 
and in fact new sequences of vector subspaces satisfying more complicated conditions are introduced in order to obtain the desired definition. 

We remark that the notion of multi-Koszul algebras for algebras having relations in two degrees is completely different from the notion of $(p,q)$-Koszul rings given in \cite{BBK}. 
On the other hand, it is easily seen that any multi-Koszul algebra with relations in degrees $2$ and $d > 2$ satisfies the $2$-$d$-Koszul property defined in \cite{GM}. 

Throughout this article $k$ will denote a field, and all vector spaces will be over $k$. 
Moreover, $V$ will always be a finite dimensional vector space, and $A$ a nonnegatively graded connected (associative) algebra over $k$ (with unit), 
to which we will usually just refer as an algebra, with irrelevant ideal $A_{>0} = \bigoplus_{n>0} A_{n}$. 
The vector space spanned by a set of elements $\{ v_{s} : s \in S \}$, for some index set $S$, will be denoted by $\mathrm{span}_{k}\cl{ v_{s} : s \in S }$ 
and the ideal $I$ generated by a set of elements $\{ \alpha_{s} : s \in S \}$ of an algebra $A$ will be denoted by $\cl{\alpha_{s} : s \in S}$. 
We will however also write the former vector space as $\cl{ v_{s} : s \in S }$ to simplify the notation, when we believe that it clearly denotes a vector space, 
but if necessary we shall stress that we mean the vector space spanned by those elements, and not the ideal generated by them, to avoid confusion. 
All unadorned tensor products $\otimes$ will be considered over $k$, unless otherwise stated. 
We shall typically denote the vector space $V^{\otimes n}$ by $V^{(n)}$, for $n \in \ZZ$, where we follow the convention $V^{(n)} = 0$, if $n < 0$, and $V^{(0)} = k$, 
and an elementary tensor $v_1\otimes \dots \otimes v_n\in V^{(n)}$ will be usually written by $v_1 \dots v_n$. 

%%%%%%%%%%%%%%%%%%%%%%%%%%%%%%%%%%%%%%%%%%%%%%%%%%%%%%%%%%%%%%%%%%%%%%%%%%%%%%%%%%%%%%%%%%%%%%%%%%%%%%%%%%%%%%%%%%%%%%%%%%%%%%%%%%%%%%%%%%%%%%%%%%%%%%%%%%%%

\section{\texorpdfstring{Preliminaries and basic properties}{Preliminaries and basic properties}}
\label{sec:preliminaries}

In this section we shall recall some basic facts about the category of (bounded below) $\ZZ$-graded modules over a nonnegatively graded connected algebra. 
We refer to \cite{C}, Exp. 15, or \cite{B4} for all the proofs of the mentioned results.  

Throughout this section, $A$ will always denote a nonnegatively graded connected algebra of the form $A=\bigoplus_{m \in \NN_{0}} A_{m}$, 
and we shall follow the typical convention $A_{m} = 0$, for $m < 0$.
A \emph{$\ZZ$-graded left (resp., right) $A$-module} $M=\bigoplus_{n \in \ZZ} M_n$ is a $\ZZ$-graded vector space together with an left (resp., right) action of $A$ on $M$ 
such that $A_m M_n \subseteq M_{m+n}$ (resp. $M_m A_{n} \subseteq M_{m+n}$), and we shall sometimes refer to them simply as \emph{left (resp., right) $A$-modules}. 
Moreover, since we will mostly deal with left modules, we will usually omit the adjective and call them just \emph{modules} (or \emph{graded modules}), if it is clear from the context, 
but we will restore it if it is necessary.
We will denote by $A$-$\grMod$ the abelian category of $\ZZ$-graded left $A$-modules, where the morphisms are the $A$-linear maps preserving the grading. 
The space of morphisms in this category between two graded $A$-modules $M$ and $M'$ will be denoted by $\hom_{A}(M,M')$. 
This category is provided with a \emph{shift functor} $(\place)[1]$ defined by $(M[1])_{n} = M_{n+1}$, where the underlying $A$-module structure of $M[1]$ is the same as the one of $M$, 
and the action of the morphisms is trivial. 
We shall also denote $(\place)[d]$ the $d$-th iteration of the shift functor. 
The $A$-module $M$ is said to be \emph{left bounded}, or also \emph{bounded below}, if there exists an integer $n_{0}$ such that $M_n=0$ for all $n<n_{0}$. 
Notice that the graded left $A$-modules which are left bounded form a full exact subcategory of $A$-$\grMod$.

If $M$ is a graded left $A$-module, we may consider the graded right $A$-module $M^{\#}$, called the \emph{graded dual}, 
which has $n$-th homogeneous component $(M_{-n})^{*}$, where $(\place)^{*}$ denotes the usual dual vector space operation, 
and if $\alpha \in A_{m}$ and $f \in (M^{\#})_{n}$, then $f \cdot \alpha \in (M^{\#})_{m+n} = (M_{-m-n})^{*}$ is defined by $(f \cdot \alpha)(x) = f(\alpha \cdot x)$, for all $x \in M_{-m-n}$. 
There exists an obvious definition if we start with a graded right $A$-module. 
We will also consider the analogous graded dual construction $(\place)^{\#}$ in the category of graded vector spaces. 
Given graded $A$-modules $N$ and $N'$, we recall the following notation:
\[     \mathcal{H}om_A (N,N') = \bigoplus\limits_{d\in \mathbb{Z}} \hom_A (N,N'[d]).     \]
We remark that, if $N$ is finitely generated, then $\mathcal{H}om_A (N,N') = \Hom_A (N,N')$, 
where the last morphism space is the usual one for $A$-modules by forgetting the gradings (see \cite{NV}, Cor. 2.4.4).  

We say that an object $M$ in $A$-$\grMod$ is \emph{$s$-concentrated} in degrees $l_1, \cdots ,$ $l_s$ if there exist integers $l_1< \cdots <l_s$ and vector subspaces of $M$, $M_{l_1}, \cdots , M_{l_s}$, such that $M=M_{l_1}\oplus \cdots \oplus M_{l_s}$.
An object $M$ in $A$-$\grMod$ is called \emph{$s$-pure} in degrees $l_1, \cdots ,$ $l_s$ if there exist integers $l_1< \cdots <l_s$ and graded vector subspaces $M_{l_1}, \cdots , M_{l_s}$ of $M$ where each $M_{l_i}$ is concentrated in degree $l_i$, such that $M=AM_{l_1}+ \cdots + AM_{l_s}$ and $M_{l_i}\cap (AM_{l_1}+\cdots +AM_{l_{i-1}})=0$ for all $i=2,\dots, s$. 
In this case, there exists an isomorphism of graded vector spaces $k \otimes_{A} M \simeq \bigoplus_{i=1}^sM_{l_i}$.
If $s=1$ we simply say that $M$ is a concentrated (respectively, pure) module (\textit{cf.} \cite{B2}).

In both cases, the integers $l_1, \cdots ,l_s$ such that $M_{l_1},\cdots ,M_{l_s}$ are nonzero are uniquely determined whenever $M$ is a nontrivial module. 
It is clear that every module which is $s$-concentrated in degrees $l_1,\cdots ,l_s$ is $s$-pure in degrees $l_1,\cdots ,l_s$, and that 
every module $s$-concentrated in degrees $l_1,\cdots ,l_s$ is isomorphic to a direct sum of shifts 
$k[-l_1]^{\dim M_{l_1}}\oplus \cdots \oplus k[-l_s]^{\dim M_{l_s}}$ as graded vector spaces. 

The following result is the graded version of the Nakayama Lemma.
%%%%%%%
\begin{lemma} 
\label{prop:trivialmod}
Let $M$ be a left bounded $\mathbb{Z}$-graded left $A$-module. 
If $k \otimes_A M=0$ then $M$ is also trivial.
\end{lemma}
%%%%%%%
\begin{proof}
See \cite{B4}, Lemme 1.3 (see also \cite{C}, Exp. 15, Prop. 6). 
\end{proof}

The $A$-module $M$ is said to be \emph{graded-free} if it is isomorphic to a direct sum of shifs $A[-l_i]$ of $A$. 
We remark that a bounded below $\ZZ$-graded $A$-module $M$ is graded-free if and only if its underlying module (\textit{i.e.} forgetting the grading) is free, 
if and only if it is projective (as a graded module or not), if and only if $\Tor_{\bullet}^{A}(k,M) = 0$, for all $\bullet \geq 1$ (or just $\bullet = 1$). 
This will follow from the comments on projective covers. 
Furthermore, it is easy to see that the graded dual of a projective graded left (resp., right) $A$-module is an injective graded right (resp., left) $A$-module (see \cite{C}, Exp. 15, Prop. 1). 

A surjective morphism $f : M \rightarrow M'$ in $A$-$\grMod$ is called \emph{essential} if for each morphism $g : N \rightarrow M$ in $A$-$\grMod$ such that $f\circ g$ is surjective, then $g$ is also surjective. 
As an application of the Nakayama Lemma we have the following result which characterizes essential surjective maps. 
%%%%%%%
\begin{lemma}
\label{lem:biyectiontensor}
Let $f : M \rightarrow M'$ be a morphism in the category $A$-$\grMod$.
\begin{itemize}
\item[\textit{(i)}] Suppose that $M'$ is left bounded and that $f$ is surjective and essential. 
                          Then $1_k\otimes_A f$ is bijective. 
                          Moreover, if $M$ is also left bounded, the converse holds.
\item[\textit{(ii)}] Assume that $f$ is surjective and $M$ is pure in degree $l$. 
                           Then $f$ is essential if and only if $f_{l} : M_{l} \rightarrow M'_{l}$ is injective.
\end{itemize}
\end{lemma}
%%%%%%%
\begin{proof}
For the first item, see \cite{B4}, Lemme 1.5 (see also \cite{C}, Exp. 15, Prop. 7). 
For the second one, see \cite{B2}, Prop. 2.4. 
\end{proof}

In fact, the last item of the previous lemma may be generalized as follows.
%%%%%%%
\begin{proposition}
\label{prop:essentialineachdegree}
Let $f\in \hom_{A}(M,M')$ be a surjective morphism. 
If $M$ is $s$-pure in degrees $l_1,\cdots ,l_s$, then $M'$ is also $s$-pure in degrees $l_1,\cdots ,l_s$. 
Moreover, $f$ is essential if and only if the induced morphisms $f_{l_i}:M_{l_i}\rightarrow M'_{l_i}$ are injective for $1\le i \le s$.
\end{proposition}
%%%%%%%
\begin{proof}
Let $m'\in M'$ and $m\in M$ be such that $f(m)=m'$. 
Since $M$ is $s$-pure in degrees $l_1,\cdots ,l_s$, there exist $\alpha_i \in A$ and $m_i\in M_{l_i}$ for $1\le i \le s$ such that $m= \sum_{i=1}^{s} \alpha_i m_i$. 
Therefore, $m'=\sum_{i=1}^{s} \alpha_i f(m_i)$, where $f(m_i)\in M'_{l_i}$. 
Thus, $M'$ is $s$-pure in degrees $l_1,\cdots ,l_s$.
By Lemma \ref{lem:biyectiontensor}, $f$ is essential if and only if the induced $k$-linear map
\[
1_k \otimes_A f : k \otimes_A M \longrightarrow k \otimes_A M'
\]
is bijective. However, since $M$ and $M'$ are $s$-pure then $k\otimes_A M$ and $k\otimes_A M'$ are canonically isomorphic to $k\otimes M_{l_1}+ \cdots + k\otimes M_{l_s}$ 
and $k\otimes M'_{l_1}+ \cdots +k\otimes M'_{l_s}$, respectively. 
Thus, the restrictions to each degree of $1_k \otimes_A f$ become $f_{l_i}:M_{l_i}\rightarrow M'_{l_i}$ for $1\le i\le s$. 
\end{proof}

Let $M$ be a nontrivial object in $A$-$\grMod$. 
A \emph{projective cover} of $M$ is a pair $(P,f)$ such that $P\in$ $A$-$\grMod$ is projective and $f:P\rightarrow M$ is an essential surjective morphism.
We remark that every left bounded $\mathbb{Z}$-graded $A$-module $M$ has a projective cover, which is unique up to (noncanonical) isomorphism (\textit{cf.} \cite{C}, Exp. 15, Th\'eo. 2).
Moreover, given $M$ a bounded below $A$-module, a projective cover may be explicitly constructed as follows. 
Since $M \neq 0$, the Nakayama lemma tells us that $M/(A_{>0} \cdot M) \simeq k \otimes_{A} M$ is a nontrivial graded vector space. 
Consider a section $s$ of the canonical projection $M \rightarrow M/(A_{>0} \cdot M) \simeq k \otimes_{A} M$. 
Now, we define $P = A \otimes (k \otimes_{A} M)$ together with the $A$-linear map $f : P \rightarrow M$ given by $f(\alpha \otimes v) = \alpha s(v)$, 
for $\alpha \in A$ and $v \in k \otimes_{A} M$. 
Using the previous lemma one directly gets that $(P,f)$ is a projective cover of $M$. 

We recall that a (graded) projective resolution $(P_{\bullet},d_{\bullet})$ of a graded $A$-module $M$ is \emph{minimal} if 
$d_{0} : P_{0} \rightarrow M$ is a projective cover (or equivalently, it is essential) and each of the maps $P_{i} \rightarrow \Ker(d_{i-1})$ induced  by $d_{i}$ is also essential, 
for all $i \in \NN$. 
We want to remark the important fact that, by iterating the process of considering projective covers for bounded below modules, 
one may easily prove that any bounded below graded $A$-module has a minimal projective resolution (see \cite{B4}, Th\'eo. 1.11).
If the $A$-module $M$ has a minimal projective resolution $(P_{\bullet},d_{\bullet})$, for any other projective resolution $(Q_{\bullet},d'_{\bullet})$ of $M$, there exists an 
isomorphism of (augmented) complexes $Q_{\bullet} \simeq P_{\bullet} \oplus H_{\bullet}$, where $H_{\bullet}$ is acyclic (see \cite{B4}, Prop. 2.2). 
Additionally, the assumption on the minimality of the projective resolution implies that the differential of the induced complex $k \otimes_{A} P_{\bullet}$ vanishes (see \cite{C}, Exp. 15, Prop. 10, 
or \cite{B4}, Prop. 2.3), so if $(P_{\bullet},d_{\bullet})$ denotes such a minimal projective resolution, one also easily gets that $P_{\bullet} \simeq A \otimes \Tor_{\bullet}^{A} (k,M)$. 
Combining the results of the two previous sentences, it is trivial to see that if $(Q_{\bullet},d'_{\bullet})$ is a projective resolution of a graded $A$-module $M$ 
having a minimal projective resolution, then the former is minimal if and only if the induced differential of $k \otimes_{A} Q_{\bullet}$ vanishes. 
Projective resolutions in this context formed of finitely generated modules are often referred to as \emph{pure projective resolutions}.

If $N$ is left bounded, let $(P_{\bullet},d_{\bullet})$ be a (minimal) graded projective resolution of $N$. 
As usual, we denote
\begin{align*}
\mathcal{E}xt_A^i (N,N') & = H^i (\mathcal{H}om_A (P_{\bullet},N')),
\\
\ext_A^i (N,N') & = H^i (\hom_A (P_{\bullet}, N')).
\end{align*}
Note that $\ext_A^i (N,N') = \mathcal{E}xt_A^i (N,N')_{0}$ and that if the projective resolution of $N$ is composed of finitely generated projective $A$-modules, 
there is a canonical identification $\mathcal{E}xt_A^{\bullet} (N,N') \simeq \Ext_A^{\bullet} (N,N')$. 
Moreover, using a very simple duality argument one can see that, if $M$ is a bounded below graded left $A$-module, then there is a canonical isomorphism of graded vector spaces 
\begin{equation}
\label{eq:isotorext}
     \mathcal{E}xt_{A}^{i}(M,k) \simeq \Tor^{A}_{i}(k,M)^{\#},
\end{equation}   
for all $i \in \NN_{0}$ (see \cite{B4}, Eq. (2.15), but \textit{cf.} also \cite{C}, Exp. 15, Prop. 2). 

We end this section by stating the beginning of the minimal projective resolution of the trivial left $A$-module $k$ for any nonnegatively graded connected algebra $A$. 
The analogous statements for the trivial right $A$-module $k$ are immediate. 
We know that the minimal projective resolution of the trivial (left) $A$-module $k$ starts as 
\[
A \otimes V \overset{\delta_1}{\longrightarrow} A \overset{\delta_0}{\longrightarrow} k \longrightarrow 0,
\]
where $\delta_0$ is the augmentation of the algebra $A$, $V \simeq A_{>0}/(A_{>0} \cdot A_{>0})$ is a vector space spanned by a minimal set of (homogeneous) generators of $A$, 
and $\delta_1$ is the restriction of the product of $A$ (see \cite{C}, Exp. 15, end of Section 7). 
This implies that $A$ is in fact a quotient of the tensor algebra $T(V)$. 
Furthermore, if we set $A = T(V)/I$, for $I$ a homogeneous ideal, it is also well-known (and follows easily from the definition) that $\Ker(\delta_{1}) \simeq I/(I \otimes V)$ 
(as graded vector spaces), so there is an isomorphism of graded vector spaces $k \otimes_{A} \Ker(\delta_{1}) \simeq I/(T(V)_{>0} \otimes I + I \otimes T(V)_{>0})$. 
A \emph{space of relations} $R$ of $A$ is defined to be a graded vector subspace of $I$ which is isomorphic to 
$I/(T(V)_{>0} \cdot I + I \cdot T(V)_{>0})$ under the canonical projection. 
Notice that its Hilbert series is thus uniquely determined, and the same holds for its first nonvanishing homogeneous component. 
It is trivially verified that the ideal of $T(V)$ generated by $R$ coincides with $I$. 
Hence, if $R$ is a space of relations of $A$, we have an isomorphism of graded vector spaces $k \otimes_{A} \Ker(\delta_{1}) \simeq R$ 
(see \cite{Go}, Lemma 1, for complete expressions of the graded vector spaces $\Tor_{\bullet}^{A}(k,k)$, for $\bullet \in \NN_{0}$, in terms of $I$ and the irrelevant ideal $T(V)_{>0}$). 
Hence, $A \otimes R \rightarrow \Ker(\delta_{1})$ is a projective cover, and the beginning of the minimal projective resolution of the trivial $A$-module $k$ is of the form 
\[
A \otimes R \overset{\delta_2}{\longrightarrow} A \otimes V \overset{\delta_1}{\longrightarrow} A \overset{\delta_0}{\longrightarrow} k \longrightarrow 0,
\]
where $\delta_2$ is induced by the usual map $\alpha \otimes v_{1} \dots v_{n} \mapsto \alpha v_{1} \dots v_{n-1} \otimes v_{n}$. 

%%%%%%%%%%%%%%%%%%%%%%%%%%%%%%%%%%%%%%%%%%%%%%%%%%%%%%%%%%%%%%%%%%%%%%%%%%%%%%%%%%%%%%%%%%%%%%%%%%%%%%%%%%%%%%%%%%%%%%%%%%%%%%%%%%%%%%%%%%%%%%%%%%%%%%%%%%%%%%%

\section{\texorpdfstring{Multi-Koszul algebras}{Multi-Koszul algebras}}
\label{sec:multikoszul}

From now on, $A$ will always denote a finitely generated nonnegatively graded connected algebra generated in degree $1$. 
This means that there exists a finite dimensional vector space $V$ considered to be in degree $1$ and a surjective morphism of graded algebras of the form $T(V) \rightarrow A$, so $A \simeq T(V)/I$, where $I \subseteq T(V)$ is a homogeneous ideal of $T(V)$. 
To avoid redundancy we will always assume that the vector space $V$ is canonically isomorphic to $A_{>0}/(A_{>0} \cdot A_{>0})$ (as graded vector spaces). 
Let us denote by $R$ a space of relations of $A$. 
We remark that in this situation, it may be equivalently defined as follows: 
for each $n \geq 2$, let $R_n$ be a subspace of $I_n$ such that it is supplementary to $I_{n-1} \otimes V + V \otimes I_{n-1}$. 
The graded vector space $R=\bigoplus_{n \geq 2} R_n$ is clearly a space of relations of $A$. 
We may thus suppose that $A = T(V)/\cl{R}$, for $R \subseteq T(V)_{\geq 2}$ a graded vector subspace. 
We will further assume that $A$ has a \emph{finite number of relations}, \textit{i.e.} that $R$ is finite dimensional, so there exists a finite subset $S$ of $\mathbb{N}_{\geq 2}$ such that 
$R = \oplus_{s \in S} R_{s}$ and $R_{s} \subseteq V^{(s)}$. 
We shall say that such an algebra $A=T(V)/I$ is \emph{$S$-multi-homogeneous}, if we want to stress the degrees of the space of relations of $A$. 

We remark however that the set $S \subseteq \NN_{\geq 2}$ may not be completely determined by $A$. 
If $A = T(V)/\cl{R}$ is $S$-multi-homogeneous, so $R = \oplus_{s \in S} R_{s}$, 
we could consider any $S' \supseteq S$ included in $\NN_{\geq 2}$, and by writing $R = \oplus_{s' \in S'} R_{s'}$, where $R_{s'}=0$ for $s' \in S' \setminus S$, 
we may also say that $A$ is $S'$-multi-homogeneous. 
Nonetheless, it is easy to see that the family of all subsets $S \subseteq \NN_{\geq 2}$ such that $A$ is $S$-multi-homogeneous contains a unique minimal element 
$\bar{S}$ given by all $s \in \NN_{\geq 2}$ such that $R_{s} \neq 0$. 

The previous considerations tell us that, under the previous assumptions, we may then work in the following concrete setting.  
We consider a graded algebra $A$ of the form $T(V)/\cl{R}$, where $V$ is a finite dimensional vector space, considered to be concentrated in degree $1$, and $R = \oplus_{s \in S} R_{s}$, 
where $R_{s} \subseteq V^{(s)}$ and $S \subseteq \NN_{\geq 2}$ is a finite subset. 
We shall suppose from now on that $A$ satisfies these hypotheses, unless otherwise stated. 
However, we will sometimes repeat (some of) the assumptions for emphasis.  
 
Note that the fact that $R$ is a space of relations of $A$ implies the following \emph{minimality condition}:
\begin{equation}\label{eq:minimality}
R_s\bigcap \Big(\sum_{s \in S} \sum_{j=0}^{s-t} V^{(j)}\otimes R_t\otimes V^{(s-t-j)}\Big)=0.
\end{equation}
Conversely, given $R = \oplus_{s \in S} R_{s}$ satisfying \eqref{eq:minimality}, where $R_{s} \subseteq V^{(s)}$ and $S \subseteq \NN_{\geq 2}$ is a finite subset, 
then $R$ is a relation space of the algebra $T(V)/\cl{R}$.  

The two-sided graded ideal $I=\bigoplus_{n\in \mathbb{Z}} I_n$ generated by $R$ in the tensor algebra $T(V)$ may be explicitly presented as 
\begin{align*}
I_n &=\sum_{s \in S} \sum_{j=0}^{n-s} V^{(j)}\otimes R_s \otimes V^{(n-s-j)}.
\end{align*}

The homogeneous components of the algebra $A$ are thus given by the vector spaces $A_n = V^{(n)}/I_n$, for $n\in \mathbb{N}_0$, and zero for $n<0$. 

For each $s \in \mathbb{N}_{\geq 2}$, we consider the map $n_s:\mathbb{N}_{0}\rightarrow \mathbb{N}_{0}$ of the form 
\[
n_s(2l)=sl, \hskip 0.4cm n_s(2l+1)=sl+1.
\]

Notice that $n_s(t+2)=n_s(t)+s$, for all $t \in \mathbb{N}_0$, and 
\[     n_s(t+1) - n_s(t) = \begin{cases} 
                                            1, &\text{if $t$ is even},
                                            \\
                                            s - 1, &\text{if $t$ is odd}.
                                     \end{cases}
\]
We shall use these elementary properties, specially in Section \ref{sec:2grados}, without further mention. 

If $s\in S$, we will denote
\[     J_{i}^{s} = \bigcap_{j=0}^{n_{s}(i)-s} V^{(j)} \otimes R_{s} \otimes V^{(n_{s}(i)-s-j)},     \]
for $i \geq 2$, and $J_i^{s} = V^{(i)}$, for $i = 0, 1$. 
We remark that the minimality condition \eqref{eq:minimality} implies that 
$J_{i}^{s}\cap (V^{(j)} \otimes J_{i'}^{s'} \otimes V^{(n_s(i)-n_{s'}(i')-j)})=0$,
for all different $s,s'\in S$, $j = 0, \dots, n_s(i)-n_{s'}(i')$, and $i, i' \in \NN_{\geq 2}$ such that $n_{s'}(i') \le n_s(i)$. 

Moreover, we define 
\[      
J_i = \bigoplus_{s \in S}  J_{i}^{s}, 
\]
if $i \geq 2$, and $J_i = V^{(i)}$, if $i = 0, 1$. 
Note that $J_{2} = R$. 

%%%%%%%
\begin{definition}
\label{def:multikoszul}
Let $A$ be an $S$-multi-homogeneous algebra with space of relations $R = \oplus_{s \in S} R_s$, $R_s \subseteq V^{(s)}$, and $S$ a finite subset of $\NN_{\ge 2}$.
We suppose the usual minimality condition on $R$. 
The \emph{left multi-Koszul complex} $(K(A)_{\bullet},\delta_{\bullet})$ of $A$ is defined by $K(A)_0 = A$, $K(A)_1 = A \otimes V$ and
$K(A)_i = A \otimes J_{i}$ for $i \geq 2$, together with the differential $\delta_{\bullet}$ where $\delta_1$ is induced by the multiplication on A, and, for $i \geq 2$,  
\[
\delta_{i} : A \otimes J_{i} \rightarrow A \otimes J_{i-1}
\]
is given by the restriction of the map $\hat{\delta}_{i} : A \otimes (\oplus_{s \in S} V^{(n_{s}(i))}) \rightarrow A \otimes (\oplus_{s \in S} V^{(n_{s}(i-1))})$, where 
\[
\hat{\delta}_i(\alpha \otimes v_{j_1} \cdots v_{j_{n_s(i)}})=\begin{cases}
\alpha v_{j_1}\cdots v_{j_{s-1}} \otimes v_{j_s}\cdots v_{j_{n_s(i)}}, & \text{if $i$ is even,}
\\
\alpha v_{j_1}\otimes v_{j_2}\cdots v_{j_{n_s(i)}}, & \text{if $i$ is odd,}
\end{cases}
\]
for $s\in S$. 
Notice that $\delta_{i}(A \otimes J_{i}^{s}) \subseteq A \otimes J_{i-1}^{s}$, for $i \geq 3$ and $s \in S$. 
It is clear that $\delta_{i+1} \circ \delta_{i} = 0$, for $i \in \NN$. 
We may also consider this complex together with the augmentation $\delta_{0} : K(A)_{0} \rightarrow k$ given by the augmentation of the algebra $A$, which we may depict as follows
\begin{equation*}
\cdots \rightarrow K(A)_i \overset{\delta_i}{\rightarrow} K(A)_{i-1}\rightarrow \cdots \rightarrow K(A)_1 \overset{\delta_1}{\rightarrow} K(A)_0 \overset{\delta_{0}}{\rightarrow} k \rightarrow 0.
\end{equation*}
Note also that the left multi-Koszul complex of $A$ is composed of graded-free left $A$-modules, and the differentials are $A$-linear maps preserving the degree.  

We say that $A$ is \emph{left multi-Koszul} if the (augmented) left multi-Koszul complex of $A$ provides a projective resolution of the trivial left $A$-module $k$, 
and in this case we may call the complex the \emph{left multi-Koszul resolution} for $A$.
\end{definition}
%%%%%%%

We remark that the left multi-Koszul complex of $A$ coincides with the minimal projective resolution of the left module $k$ 
seen at the end of Section \ref{sec:preliminaries} up to homological degree $2$. 
It is clear that the left multi-Koszul resolution for $A$ is minimal (because the induced differential of the complex $k \otimes_{A} K(A)_{\bullet}$ vanishes) and projective.
It is straightforward to see that an algebra is left multi-Koszul if and only if its left multi-Koszul complex defined above is acyclic in positive homological degrees. 

%%%%%%%
\begin{remark}
\label{rem:bdef}
Since an algebra $A$ may be regarded to be $S$-multi-homogeneous for different subsets $S \subseteq \NN_{\geq 2}$, one may wonder whether the definition of left multi-Koszul algebra actually depends on the subset $S$. 
It is however trivially verified that this is not the case, \textit{i.e.} if $A$ is regarded as $S$-multi-homogeneous and also $S'$-multi-homogeneous, then it is 
left $S$-multi-Koszul if and only if it is left $S'$-multi-Koszul. 
The same phenomenon also occurs for the definition of right multi-Koszul property presented in the following remark. 
\end{remark}
%%%%%%%

\begin{remark}
\label{rem:rightmultikoszul}
There is also an analogous definition of \emph{right multi-Koszul complex}  and hence of \emph{right multi-Koszul algebra}. 
Using the same notation as in the previous definition, the right multi-Koszul complex $(K(A)'_{\bullet},\delta'_{\bullet})$ of $A$
is defined by $K(A)'_{0} = A$, $K(A)_{1} = V \otimes A$ and $K(A)'_{i} = J_{i} \otimes A$, for $i \geq 2$, together with the differential 
$\delta'_{\bullet}$ where $\delta'_{1}$ is induced by the multiplication on A, and, for $i \geq 2$, 
\[
\delta'_{i} : J_{i} \otimes A \rightarrow J_{i-1} \otimes A
\]
is given by the restriction of the map $\hat{\delta}'_{i} : (\oplus_{s \in S} V^{(n_{s}(i))}) \otimes A \rightarrow (\oplus_{s \in S} V^{(n_{s}(i-1))}) \otimes A$, where 
\[
\hat{\delta}'_i(v_{j_{n_s(i)}} \cdots v_{j_{1}} \otimes \alpha)=\begin{cases}
v_{j_{n_s(i)}}\cdots v_{j_{s}} \otimes v_{j_{s-1}}\cdots v_{j_{1}} \alpha, & \text{if $i$ is even,}
\\
v_{j_{n_s(i)}}\cdots v_{j_2} \otimes  v_{j_{1}} \alpha, & \text{if $i$ is odd,}
\end{cases}
\]
for $s\in S$. 
We will also consider this complex together with the augmentation $\delta'_{0} : K(A)'_{0} \rightarrow k$ given by the augmentation of the algebra $A$. 

Notice that the right multi-Koszul complex of $A$ coincides with the minimal projective resolution of the right module $k$ 
mentioned at the end of Section \ref{sec:preliminaries} up to homological degree $2$. 

We say that $A$ is \emph{right multi-Koszul} if the (augmented) right multi-Koszul complex of $A$ provides a projective resolution of the trivial right $A$-module $k$, 
and in this case we may call the complex the \emph{right multi-Koszul resolution} for $A$.

It is easy to check that the right multi-Koszul resolution for $A$ is minimal (because the induced differential of the complex $K(A)'_{\bullet} \otimes_{A} k$ vanishes) and projective, 
and that an algebra is right multi-Koszul if and only if its right multi-Koszul complex defined above is acyclic in positive homological degrees. 
\end{remark}
%%%%%%%

\begin{remark}
Note that the previous definition of left or right multi-Koszul property coincides with the corresponding one given in \cite{BDW}, Section 5, if the algebra is homogeneous, 
and in principle not to the one given in \cite{B2}, Definition 2.10. 
Nonetheless, using \cite{BDW}, Prop. 3, one immediately deduces that for a homogeneous algebra both definitions are equivalent, implying that such an algebra is multi-Koszul 
if and only if it is generalized Koszul, and in fact the left (resp., right) multi-Koszul complex coincides with the (generalized) 
left (resp., right) Koszul complex defined by Priddy if the algebra is quadratic and by Berger if the algebra is homogeneous. 
\end{remark}
%%%%%%%

\begin{remark}
Notice that, for $d \in \NN_{> 2}$, if $A$ is a left $\{2,d\}$-multi-Koszul algebra for the previous definition, then it is in particular a $2$-$d$-Koszul algebra in the sense defined in \cite{GM}. 
The converse however does not hold (see the algebra $B$ in \cite{CG}, which is not $\{ 2, 3\}$-multi-Koszul). 
\end{remark}
%%%%%%%

Since the length of a minimal projective resolution of $k$ gives the global dimension of $A$, the following proposition is immediate.
%%%%%%%
\begin{proposition}
Let $A=T(V)/\cl{R}$ be an $S$-multi-homogeneous algebra (with $S$ a finite subset of $\mathbb{N}_{\geq 2}$) such that $R=\bigoplus_{s\in S}R_s$ 
satisfies the minimality condition. 
If the global dimension of $A$ is $2$, then $A$ is $S$-multi-Koszul.
\end{proposition}
%%%%%%%

We also have the following result, which shows a way to produce (an infinite number of) examples of multi-Koszul algebras. 
%%%%%%%
\begin{proposition}
\label{prop:libre}
Let $\{ B^{s} : s \in S \}$, where $S \subseteq \NN_{\geq 2}$, be a finite collection of nonnegatively graded connected algebras such that $B^{s}$ is $s$-Koszul, for each $s \in S$. 
Then, the free product (\textit{i.e.} the coproduct in the category of graded algebras) $A = \coprod_{s \in S} B^{s}$ of the collection $\{ B^{s} : s \in S \}$ is a multi-Koszul algebra. 
\end{proposition}
%%%%%%%
\begin{proof}
Let us suppose that $B^{s} = T(V_{s})/\cl{R_{s}}$, for $s \in S$, is an $s$-Koszul algebra, where $R_{s} \subseteq V_{s}^{(s)}$. 
By the definition of the free product of the collection $\{ B^{s} : s \in S \}$, we may consider that $A = T(V)/\cl{R}$, where $V = \oplus_{s \in S} V_{s}$, and $R = \oplus_{s \in S} R_{s}$. 
The canonical inclusion $B^{s} \hookrightarrow A$ is a morphism of graded algebras, and it makes $A$ a free graded (left or right) $B^{s}$-module. 

On the one hand, it is clear that, if $s \in S$,     
\[     J_{i}^{s} = \bigcap_{j=0}^{n_{s}(i)-s} V^{(j)} \otimes R_{s} \otimes V^{(n_{s}(i)-s-j)} =\bigcap_{j=0}^{n_{s}(i)-s} V_{s}^{(j)} \otimes R_{s} \otimes V_{s}^{(n_{s}(i)-s-j)},     \]  
for $i \geq 2$. 
We recall that $J_i = \oplus_{s \in S} J_{i}^{s}$. 
Moreover, $J_{1} = V = \oplus_{s \in S} V_{s}$, and $J_{0} = k$. 
In fact, if $(K(B^{s})_{\bullet},\delta_{\bullet}^{s})$ denotes the Koszul complex of $B^{s}$, which is acyclic in positive homological degrees by assumption, 
then $K(B^{s})_{\bullet} = B^{s} \otimes J_{\bullet}^{s}$. 
Since $A$ is free as a $B^{s}$-module, we have that $A \otimes_{B^{s}} K(B^{s})_{\bullet} = A \otimes J_{\bullet}^{s}$ is also acyclic in positive homological degrees. 
This can be proved as follows. 
We consider the convergent spectral sequence of change of base $E^{2}_{p,q} = \Tor_{p}^{B^{s}}(A,H_{q}(K(B^{s})_{\bullet} )) \Rightarrow H_{p+q}(A \otimes_{B^{s}} K(B^{s})_{\bullet})$ 
(see \cite{W}, Application 5.7.8, where we have used the fact that $K(B^{s})_{\bullet}$ is a bounded below complex of free modules). 
The exactness of the Koszul complex of $B^{s}$ and the freeness of the $B^{s}$-module $A$ imply that $E^{2}_{p,q} = 0$ if $(p,q) \neq (0,0)$. 
In consequence, $H_{n}(A \otimes_{B^{s}} K(B^{s})_{\bullet}) = 0$, for $n \geq 1$. 

Since the multi-Koszul complex $(K(A)_{\bullet},\delta_{\bullet})$ of the algebra $A$ can be decomposed as $K(A)_{\bullet} = \oplus_{s \in S} A \otimes_{B^{s}} K(B^{s})_{\bullet}$, 
for $\bullet \geq 1$, and $\delta_{\bullet} = \oplus_{s \in S} \delta_{\bullet}^{s}$ for $\bullet \geq 2$, the exactness of $A \otimes_{B^{s}} K(B^{s})_{\bullet}$ 
in positive homological degrees tells us that $K(A)_{\bullet}$ is acyclic in homological degrees greater than or equal to $2$. 
On the other hand, the exactness of the multi-Koszul complex in homological degree $1$ is automatically satisfied for a nonnegatively graded connected algebra. 
We have thus that $K(A)_{\bullet}$ is exact in positive homological degrees, so $A$ is multi-Koszul. 
\end{proof}

Let $A=T(V)/\cl{R_a}$ and $B=T(W)/\cl{R_b}$ be such that $A$ is $a$-Koszul, $B$ is $b$-Koszul, 
where we consider $V$ and $W$ to be subspaces of a fixed vector space $U$, and the minimality condition for $R_a \oplus R_b \subseteq T(U)$ is satisfied. 
Consider the $\{a,b\}$-homogeneous algebra $C=T(V\cup W)/\cl{R_a,R_b}$. 
One may wonder if the previous result could be weakened in order to obtain that $C$ is also multi-Koszul. 
As expected, the answer is no, as we may see in the following example.

%%%%%%%
\begin{example}
\label{ex:notcoprodcasi}
Consider the algebras $A=k\cl{x,z}/\cl{xz}$ and $B=k\cl{x,y}/\cl{y^2x}$. 
It is direct that $A$ is $2$-Koszul, $B$ is $3$-Koszul and that the minimality condition for the relation space $\mathrm{span}_{k} \cl{xz, y^{2}x} \subseteq k\cl{x,y,z}$ holds. 
The algebra $C=k\cl{x,y,z}/\cl{xz,y^2x}$ is not (left) multi-Koszul since $k \otimes_{C} \Ker(\delta_2) =\mathrm{span}_{k}\cl{y^2xz} \neq J_3^3$.

A similar example can be also obtained for $V = W$ in the previous notation. 
Consider the algebras $A=k\cl{x,y}/\cl{xy}$ and $B=k\cl{x,y}/\cl{y^{2}x}$, which are respectively $2$-Koszul and $3$-Koszul. 
The minimality condition for $\mathrm{span}_{k} \cl{xy, y^{2}x} \subseteq k\cl{x,y}$ is satisfied. 
The algebra $C=k\cl{x,y}/\cl{xy,y^2x}$ is however not left multi-Koszul, since $k \otimes_{C} \Ker(\delta_2) = \mathrm{span}_{k} \cl{xy^{2}x, y^2xy}$ but $J_3 = 0$.
\end{example}
%%%%%%%

Regarding some of the several different equivalent definitions of the (generalized) Koszul property for homogeneous algebras, one may wonder for instance 
if an $S$-multi-homogeneous algebra is left multi-Koszul if and only if the trivial $A$-module $k$ has a minimal projective resolution $(P_{\bullet},d_{\bullet})$ whose $i$-th projective $P_{i}$ is pure in degrees $\{n_{s}(i) : s \in S\}$, for all $i \in \NN_{0}$ 
(\textit{cf.} \cite{B2}, Prop. 2.13). 
Unfortunately, this proposal is different from the Definition \ref{def:multikoszul} (\textit{e.g.} Example \ref{ex:difkos}), and it leads to several undesirable situations. 
For instance, if one only asks that the $i$-th projective $P_{i}$ of the minimal projective resolution of the trivial $A$-module $k$ is pure in degrees $\{n_{s}(i) : s \in S\}$, for all $i \in \NN_{0}$,  
the differentials of such a minimal projective resolution may become rather complicated, since the relations may interact in complicated ways. 
Furthermore, this (alternative) definition would not be yield to a proper generalization of the (generalized) Koszul algebras, if we are willing to have such a definition 
independent of the index set $S$ in the sense of Remark \ref{rem:bdef} (see the algebra $B$ of the Example \ref{ex:loco}). 
Moreover, following the ideas of Example \ref{ex:loco}, we see that this alternative definition would also allow lots of bad behaved algebras containing ``non-Koszul components'', even in the generalized sense, and in particular, the Yoneda algebras obtained from this alternative definition will not be necessarily $\mathcal{K}_{2}$ (see the algebra $A$ of the Example \ref{ex:loco}). 
We next list these examples thus in order to understand better why Definition \ref{def:multikoszul} is in some sense reasonable to avoid the aforementioned phenomena, 
which we may regard as pathological.

%%%%%%%
\begin{example}
\label{ex:difkos}
Consider the algebra $A=k\cl{x,y,z}/\cl{xy,y^2z}$. 
One may easily compute the minimal projective resolution of the trivial $A$-module $k$ of the form
\[
0\longrightarrow A\otimes \cl{xy^2z} \overset{\tilde{\delta}_3}{\longrightarrow} A\otimes \cl{xy,y^2z} \overset{\delta_2}{\longrightarrow} A\otimes \cl{x,y,z} \overset{\delta_1}{\longrightarrow} A\overset{\delta_0}{\longrightarrow} k\longrightarrow 0,
\]
where the morphisms $\delta_i$ for $i=0, 1,2$ are defined as above, and $\tilde{\delta}_{3}(\alpha \otimes xy^{2}z) = \alpha x \otimes y^{2}z$. 
Note that the $i$-th projective module is pure in degrees $n_2(i)$ and $n_3(i)$, for all $i \in \NN_0$. 
However, the algebra is not left multi-Koszul, for the condition $\Ker(\delta_2)/(A_{> 0} \cdot \Ker(\delta_2)) = J_3^2\oplus J_3^3$ fails.
\end{example}
%%%%%%%

\begin{example}
\label{ex:loco}
Consider the algebras $B=k\cl{x,y,z}/\cl{x^2y,z^2x}$ and $C=k\cl{u}/\cl{u^4}$. 
It is easy to check that $C$ is $4$-Koszul whereas $B$ is not $3$-Koszul (for instance, because $z^2x^2y$ is a minimal generator of degree $5$ of the kernel of the second differential). 
Let $A=B*_k C$ be the free product of $B$ and $C$. 
There exists a minimal graded projective resolution
\[
\dots \rightarrow A\otimes \cl{u^{n_4(i)}} \rightarrow \cdots \rightarrow A\otimes \cl{u^8} \rightarrow A\otimes W \rightarrow A\otimes R\rightarrow A\otimes V\rightarrow A\rightarrow k\rightarrow 0,
\]
where $V = \mathrm{span}_{k}\cl{x,y,z,u}$, $R = \mathrm{span}_{k}\cl{x^2y, z^2x, u^4}$, 
$W = \mathrm{span}_{k} \cl{z^2x^2y,u^5}$ and the differentials are the obvious ones. 
This implies that the $i$-th projective module is pure in degrees $n_3(i)$ and $n_4(i)$, 
for all $i \in \NN_0$. 
However, $k\otimes_A \Ker(\delta_2) \neq J_3^3\oplus J_3^4$, so $A$ is not left multi-Koszul. 

We note incidentally that, if the algebra $B$ is regarded as an $S$-multi-homogeneous algebra for $S = \{3,4\}$ (by considering $R_{4} = 0$), it satisfies that 
the trivial left $A$-module $k$ has a minimal projective resolution whose $i$-th projective is pure in degrees $n_3(i)$ and $n_4(i)$, for all $i \in \NN_0$. 
\end{example}
%%%%%%%

Before proceeding further we want to make several comments on the left multi-Koszul complex of an $S$-multi-homogeneous algebra $A$. 
The obvious statements for the right multi-Koszul complex trivially hold. 
First, given $i \in \NN_{0}$, note that the map of graded vector spaces $J_{i+1} \rightarrow \Ker(\delta_{i})$ given by the restriction of $\delta_{i+1}$ is injective. 
This can be proved as follows. 
The cases $i =0, 1$ are immediate, so we will suppose that $i \geq 2$.  
In that case, the previous map is the direct sum of the maps $J_{i+1}^{s} \rightarrow \Ker(\delta_{i}) \cap (A \otimes J_{i}^{s})$, for $s \in S$, so the kernel 
of the former is the direct sum of the kernel of the previous maps for each $s \in S$.
Furthermore, the kernel of the restriction of $\delta_{i+1}$ to $J_{i+1}^{s}$ is easily seen to be $J_{i+1}^{s} \cap (I_{n_{s}(i+1)-n_{s}(i)} \otimes J_{i}^{s})$. 
The first term of this intersection is included in $R_{s} \otimes V^{(n_{s}(i-1))}$, whereas the second is included in 
\[     \big(\sum_{s'<s} \sum_{j=0}^{s-s'} V^{(j)} \otimes R_{s'} \otimes V^{(s-s'-j)}\big) \otimes V^{(n_{s}(i-1))},     \]
where we have used that $n_{s}(i+1)-n_{s}(i) < s$. 
The intersection of the last two vanishes by the minimality condition \eqref{eq:minimality}, so \textit{a fortiori} the kernel of the restriction of $\delta_{i+1}$ to $J_{i+1}^{s}$ vanishes, 
which in turn implies the injectivity of the mentioned morphism $J_{i+1} \rightarrow \Ker(\delta_{i})$. 

For each $i \in \NN_{0}$, let us now consider the map $J_{i+1} \rightarrow k \otimes_{A} \Ker(\delta_{i})$ given by the composition of $J_{i+1} \rightarrow \Ker(\delta_{i})$ 
and the canonical projection $\Ker(\delta_{i}) \rightarrow k \otimes_{A} \Ker(\delta_{i})$. 
We claim that this composition is in fact injective if $i$ is even. 
This can be proved as follows. 
By the comments at the end of Section \ref{sec:preliminaries}, we know that the mentioned map is in fact an isomorphism for $i = 0$ (and also for $i = 1$). 
We shall suppose thus that $i \geq 2$. 
As before, the mentioned map can be decomposed as the direct sum of the corresponding maps of the form $J_{i+1}^{s} \rightarrow k \otimes_{A} (\Ker(\delta_{i}) \cap (A \otimes J_{i}^{s}))$, 
for $s \in S$. 
Hence, it suffices to prove the injectivity of each of these components. 
Since $i$ is even, the image of the map $J_{i+1}^{s} \rightarrow \Ker(\delta_{i}) \cap (A \otimes J_{i}^{s})$ is contained in $V \otimes J_{i}^{s}$, 
so one sees that the kernel of $J_{i+1}^{s} \rightarrow k \otimes_{A} (\Ker(\delta_{i}) \cap (A \otimes J_{i}^{s}))$ 
vanishes if and only if $J_{i}^{s} \cap \Ker(\delta_{i}) = 0$, which follows from the previous paragraph. 

On the other hand, if $i \geq 3$ is odd, we may state the following fact about the map $J_{i+1} \rightarrow k \otimes_{A} \Ker(\delta_{i})$. 
We know that it can be decomposed as the direct sum of the maps of the form 
$J_{i+1}^{s} \rightarrow k \otimes_{A} (\Ker(\delta_{i}) \cap (A \otimes J_{i}^{s}))$, for $s \in S$. 
It suffices thus to analyse each of these components separately. 
For $i$ odd, the image of the map $J_{i+1}^{s} \rightarrow \Ker(\delta_{i}) \cap (A \otimes J_{i}^{s})$ is now contained in $V^{(s-1)} \otimes J_{i}^{s}$,
so one sees that the kernel of $J_{i+1}^{s} \rightarrow k \otimes_{A} (\Ker(\delta_{i}) \cap (A \otimes J_{i}^{s}))$ 
is nontrivial if and only if there exists $t \leq s - 2$ such that $(V^{(t)} \otimes J_{i}^{s}) \cap \Ker(\delta_{i}) \neq 0$, which is equivalent to say that the graded vector space 
$k \otimes_{A} (\Ker(\delta_{i}) \cap (A \otimes J_{i}^{s}))$ has nontrivial homogeneous components of degree strictly less than $n_{s}(i+1)$. 

This proves the following result.
%%%%%%%
\begin{lemma}
\label{lemma:tresimp}
Let $A$ be an $S$-multi-homogeneous algebra, with space of relations $R = \oplus_{s \in S} R_s$, $R_s \subseteq V^{(s)}$, and $S$ a finite subset of $\NN_{\ge 2}$, 
and let $(K(A)_{\bullet},\delta_{\bullet})$ be its left multi-Koszul complex. 
Given $i \in \NN_{0}$, the map of graded vector spaces $J_{i+1} \rightarrow \Ker(\delta_{i})$ given by the restriction of $\delta_{i+1}$ is injective. 
Consider now the map of graded vector space given by the composition of the previous morphism and the canonical projection $\Ker(\delta_{i}) \rightarrow k \otimes_{A} \Ker(\delta_{i})$. 
If $i$ is even or $i = 1$, it is injective, and if $i$ is odd and $i \geq 3$, it is injective if and only if the graded vector space $k \otimes_{A} (\Ker(\delta_{i}) \cap (A \otimes J_{i}^{s}))$ 
has no nontrivial homogeneous components of degree strictly less than $n_{s}(i+1)$, for all $s \in S$.
\end{lemma}
%%%%%%%

The corresponding formulation of the lemma for the right multi-Koszul complex of $A$ is obvious, and we shall refer to the lemma whether we are considering the left or the right version. 
We may use the previous lemma in fact to prove the first main result of this section (\textit{cf.} \cite{BDW}, Prop. 3):
%%%%%%%
\begin{proposition}
\label{prop:tresimp}
Let $A$ be an $S$-multi-homogeneous algebra with space of relations $R = \oplus_{s \in S} R_s$, $R_s \subseteq V^{(s)}$, and $S$ a finite subset of $\NN_{\ge 2}$.
Then $A$ is left (resp., right) multi-Koszul if and only if there is an isomorphism of graded vector spaces 
$\Tor_i^A(k,k) \simeq J_{i}$, for all $i \in \NN_{0}$.
\end{proposition}
%%%%%%%
\begin{proof}
We shall prove the statement for the left multi-Koszul property, since the right one is analogous. 
Moreover, we will only show the ``if'' part, since the converse follows immediately from the minimality of the (left) multi-Koszul complex. 

Assume the existence of the isomorphism of graded vector spaces in the statement. 
We will prove that the left multi-Koszul complex is in fact a minimal projective resolution of the trivial left $A$-module $k$. 
In fact, we will show that $K(A)_{\bullet}$ is a minimal projective resolution of $k$ up to homological degree $i$, for all $i \in \NN$. 
Since the former coincides with such a minimal projective resolution up to homological degree $2$, we suppose that the statement is true for $i \geq 2$. 
By the comments on the construction of projective covers in Section \ref{sec:preliminaries} and the assumption $\Tor_{i+1}^A(k,k) \simeq J_{i+1}$, 
there is an essential surjective morphism of graded $A$-modules $h_{i} : A \otimes J_{i+1} \rightarrow \Ker(\delta_{i})$ inducing an isomorphism $J_{i+1} \rightarrow k \otimes_{A} \Ker(\delta_{i})$. 
We will consider two cases. 
If $i = 2$, the previous lemma tells us that the composition
\[
J_{3} \hookrightarrow \Ker(\delta_2)\twoheadrightarrow k \otimes_{A} \Ker(\delta_2),
\]
where the first map is the restriction of $\delta_{3}$, is injective.
Hence, the composition of this map with the inverse of $1_{k} \otimes_{A} h_{2}$ is an injective endomorphism of graded vector spaces of $J_{3}$, 
so an isomorphism, since the latter has finite dimension. 
This in turn implies that $\delta_{3}$ is in fact a projective cover of $\Ker(\delta_{2})$ by Lemma \ref{lem:biyectiontensor}. 

We now assume that $i \geq 3$. 
In this case, since the morphism $\delta_{i}$ is a direct sum of its components $A \otimes J^{s}_{i} \rightarrow A \otimes J^{s}_{i-1}$, for each $s \in S$, 
we see that the projective cover of $\Ker(\delta_{i})$ is the direct sum of the projective covers of each component $A \otimes J^{s}_{i} \rightarrow A \otimes J^{s}_{i-1}$, for $s \in S$, 
\textit{i.e.} it is the direct sum of the graded free $A$-modules $A \otimes k \otimes_{A} (\Ker(\delta_{i}) \cap (A \otimes J^{s}_{i}))$, for $s \in S$. 
By induction on $s$ we easily see that, for all $s \in S$, the graded vector space $k \otimes_{A} (\Ker(\delta_{i}) \cap (A \otimes J_{i}^{s}))$ 
has no nontrivial homogeneous components of degree strictly less than $n_{s}(i+1)$ and contains $J_{i+1}^{s}$, 
so by the assumption on the Tor group it coincides with $J_{i+1}^{s}$. 
We may now proceed as in the case $i = 2$, since by the previous lemma the composition
\[
J_{i+1} \hookrightarrow \Ker(\delta_i) \twoheadrightarrow k \otimes_{A} \Ker(\delta_i),
\]
where the first map is the restriction of $\delta_{i+1}$, is injective.
Hence, the composition of this map with the inverse of $1_{k} \otimes_{A} h_{i}$ is an injective endomorphism of graded vector spaces of $J_{i+1}$, 
so an isomorphism, since the latter has finite dimension. 
This in turn implies that $\delta_{i+1}$ is in fact a projective cover of $\Ker(\delta_{i})$. 
The proposition is thus proved. 
\end{proof}

We have thus the following direct consequence.
%%%%%%%
\begin{corollary}
\label{prop:left=right}
Let $A$ be an $S$-multi-homogeneous algebra with space of relations $R = \oplus_{s \in S} R_s$, $R_s \subseteq V^{(s)}$, and $S$ a finite subset of $\NN_{\ge 2}$.
Then, $A$ is left $S$-multi-Koszul if and only if it is right $S$-multi-Koszul.
\end{corollary}
%%%%%%%
By the previous result, we shall usually say that an algebra $A$ is left $S$-multi-Koszul, right $S$-multi-Koszul, $S$-multi-Koszul or simply \emph{multi-Koszul} indiscriminately.

We may restate the previous results in a slightly different manner. 
Consider the $k$-linear endomorphism of $T(V)$ given by
\begin{itemize}
\item $\tau (1)=1$,
\item $\tau(w_1\otimes w_2\otimes \cdots \otimes w_n)=w_n\otimes \cdots \otimes w_2\otimes w_1$,
\end{itemize}
for $w_1,\cdots ,w_n\in V$ and $n \geq 1$, which is an anti-isomorphism of algebras, so it induces an algebra anti-isomorphism 
\[
\bar{\tau}: A \longrightarrow \frac{T(V)}{\cl{\tau (R)}}=A^{\circ}. 
\]
In other words, it induces an isomorphism between the (usual) opposite algebra $A^{\mathrm{op}}$ of $A$ and $A^{\circ} = T(V)/\cl{\tau(R)}$. 

If the relation space $R$ satisfies the minimality condition, then the relation space $\tau(R)$ also satisfies it. 
We shall say that $A^{\circ}$ is the \emph{opposite $S$-multi-homogeneous algebra} of $A$.
%%%%%%%
\begin{corollary}
\label{prop:koszulitydual}
The algebra $A^{\circ}$ is $S$-multi-Koszul if and only if $A$ is $S$-multi-Koszul.
\end{corollary}
%%%%%%%
\begin{proof}
It is an immediate consequence of Corollary \ref{prop:left=right}.
\end{proof}

We now have the following immediate consequence of Proposition \ref{prop:tresimp}  and the isomorphism \eqref{eq:isotorext} for the Yoneda algebra of a multi-homogeneous algebra.
%%%%%%%
\begin{proposition}
\label{prop:tresimpext}
For an $S$-multi-homogeneous algebra $A$ with space of relations $R = \oplus_{s \in S} R_s$, $R_s \subseteq V^{(s)}$, and $S$ a finite subset of $\NN_{\ge 2}$, 
the following statements are equivalent:
\begin{itemize}
\item[\textit{(i)}] $A$ is multi-Koszul,
\item[\textit{(i)}] $\mathcal{E}xt_A^i (k,k) \simeq J_{i}^{\#}$, for all $i \in \NN_{0}$.
\end{itemize}
\end{proposition}
%%%%%%%

We note that the multi-Koszul resolution for a multi-Koszul algebra $A$ is composed of finitely generated projective $A$-modules, for each vector space $J_{i}$ is finite dimensional, 
so, by the comments at the end of Section \ref{sec:preliminaries}, there is a canonical identification $\mathcal{E}xt_A^{\bullet} (k,k) \simeq \Ext_A^{\bullet} (k,k)$. 

We want now to relate the multi-Koszul property of an multi-homogeneous algebras with the generalized Koszul property of some associated homogeneous algebras. 
We will need the following auxiliary result in the sequel. 
%%%%%%%
\begin{lemma}
\label{lemma:nakacons}
Let $A$ be a nonnegatively graded connected algebra, $I \subseteq A_{> 0}$ be a homogeneous ideal, and  
\[     
L \overset{g}{\rightarrow} M \overset{f}{\rightarrow} N     
\]
be a sequence of graded (left) $A$-modules satisfying that $f \circ g = 0$. 
If the sequence
\[     
(A/I) \otimes_{A} L \overset{1_{A/I} \otimes g}{\longrightarrow} (A/I) \otimes_{A} M \overset{1_{A/I} \otimes f}{\longrightarrow} (A/I) \otimes_{A} N     
\]
is exact, then the former sequence is also exact.
\end{lemma}
%%%%%%%
\begin{proof}
We only need to prove that $\Ker(f) \subseteq \im(g)$, since the other inclusion follows from $f \circ g = 0$. 
Moreover, it is also clear that it suffices to show that each homogeneous component $\Ker(f)$ is included in $\im(g)$, for the morphisms are homogeneous of degree zero. 
Let $m \in M$ be a homogeneous element such that $f(m) = 0$. 
This implies that $(1_{A/I} \otimes f)(1 \otimes m) = 0$. We shall prove that there exists $l \in L$ such that $g(l) = m$. 

The exactness of the second sequence tells us that there exists $l \in L$ (homogeneous and of the same degree as $m$) such that $(1_{A/I} \otimes g)(1 \otimes l) = 1 \otimes m$, 
\textit{i.e.} $1 \otimes (m - g(l)) = 0$. 
Let $M'$ be the (left) $A$-submodule of $M$ generated by $m-g(l)$. 
Notice that $M'$ is bounded below, for $m-g(l)$ is homogeneous. 
Hence, we get that $(A/I) \otimes_{A} M' = 0$, and, therefore, 
\[     
k \otimes_{A} M' = (A/A_{>0}) \otimes_{A} M' = (A/A_{>0}) \otimes_{A/I} (A/I) \otimes_{A} M' = 0.     
\]
By the Nakayama Lemma it must be $M' = 0$, so $m = g(l)$. The lemma is thus proved. 
\end{proof}

We shall now state the second main result of this section. 
As usual, $r$-$\pd_{A}(M)$ will denote the projective dimension of the right $A$-module $M$.
%%%%%%%
\begin{theorem}
\label{thm:koszulbypeaces}
Let $A = T(V)/\cl{R}$ be an $S$-multi-homogeneous algebra with space of relations $R = \oplus_{s \in S} R_s$, $R_s \subseteq V^{(s)}$, and $S$ a finite subset of $\NN_{\ge 2}$.
We denote by $A^{s} = T(V)/\cl{R_{s}}$ the associated $s$-th homogeneous algebra.
The following conditions are equivalent:
\begin{itemize}
\item[\textit{(i)}] $A$ is multi-Koszul.
\item[\textit{(ii)}] For each $s \in S$, we have that $A^{s}$ is $s$-Koszul, $r$-$\pd_{A^{s}}(A) \leq 1$, and 
$\Ker(\delta_{2}) = \oplus_{s \in S} (\Ker(\delta_{2}) \cap (A \otimes R_{s}))$, where $\delta_{2} : A \otimes R \rightarrow A \otimes V$ is the second differential of the left 
multi-Koszul complex of $A$. 
\end{itemize}
\end{theorem}
%%%%%%%
\begin{proof} 
We shall first prove the implication \textit{(i)} $\Rightarrow$ \textit{(ii)}. 
Assume that $A$ is multi-Koszul, \textit{i.e.} the multi-Koszul complex $K(A)_{\bullet}$ of $A$ is acyclic in positive homological degrees. 
Since 
\begin{equation}
\label{eq:ker}
\Ker(\delta_{2}) = \im(\delta_{3}) = \bigoplus_{s \in S} (\im(\delta_{3}) \cap (A \otimes R_{s})) = \bigoplus_{s \in S} (\Ker(\delta_{2}) \cap (A \otimes R_{s})),     
\end{equation}
we get the last condition of item \textit{(ii)}. 

We will now show that $A^{s}$ is $s$-Koszul and $r$-$\pd_{A^{s}}(A) \leq 1$, for each $s \in S$. Let $s \in S$ be a fixed index. 
Consider a subcomplex $K(A)_{\bullet}^{s}$ of the (nonaugmented) multi-Koszul complex $K(A)_{\bullet}$ of $A$, given by $K(A)_{\bullet}^{s} = A \otimes J_{\bullet}^{s}$ if $\bullet \geq 2$, 
$K(A)_{\bullet}^{s} = K(A)_{\bullet}$ if $\bullet =0,1$, provided with the induced differential of $K(A)_{\bullet}$. 
It is straightforward to check that the fact that the multi-Koszul complex $K(A)_{\bullet}$ of $A$ is acyclic in positive homological degrees 
implies that the subcomplex $K(A)_{\bullet}^{s}$ is acyclic in homological degrees greater than or equal to $3$. 
Moreover, using \eqref{eq:ker} and the exactness of $K(A)_{\bullet}$ in homological degree $2$ we also get that $K(A)_{\bullet}^{s}$ is acyclic in homological degree $2$. 
As a consequence, we conclude that $K(A)_{\bullet}^{s}$ is acyclic in homological degrees greater than or equal to $2$. 
Using the evident isomorphism $(K(A)^{s}_{\bullet},\delta_{\bullet}|_{K(A)^{s}_{\bullet}}) \simeq A \otimes_{A^{s}} (K(A^{s})_{\bullet},\delta_{\bullet}^{s})$ of complexes of graded (left) $A^{s}$-modules, 
and Lemma \ref{lemma:nakacons}, we get that $(K(A^{s})_{\bullet},\delta_{\bullet}^{s})$ is acyclic in homological degrees greater than or equal to $2$. 
Since the Koszul complex of an $s$-homogeneous algebra is always acyclic in the first homological degree, we conclude that $(K(A^{s})_{\bullet},\delta_{\bullet}^{s})$ 
is acyclic in positive homological degrees, so $A^{s}$ is $s$-Koszul. 

On the other hand, since $A^{s}$ is $s$-Koszul, $(K(A^{s})_{\bullet},\delta_{\bullet}^{s})$ is a minimal projective resolution of $k$ in the category of graded  $A^{s}$-modules. 
We recall that for a bounded below graded right $A^{s}$-module $M$, the minimal projective resolution $(P_{\bullet},d_{\bullet})$ of $M$ in the category of graded right $A^{s}$-modules is of the form $P_{\bullet} = \Tor_{\bullet}^{A^{s}}(M,k) \otimes A^{s}$. 
Since $A$ is a bounded below (right) $A^{s}$-module, and $(K(A)^{s}_{\bullet},\delta_{\bullet}|_{K(A)^{s}_{\bullet}}) \simeq A \otimes_{A^{s}} (K(A^{s})_{\bullet},\delta_{\bullet}^{s})$ 
is acyclic in homological degrees greater than or equal to $2$, 
we conclude that $A$ has a minimal projective resolution in the category of graded right $A^{s}$-modules of length at most $1$, so $r$-$\mathrm{pdim}_{A^{s}}(A) \leq 1$. 

We shall now prove the converse implication, \textit{i.e.} \textit{(ii)} $\Rightarrow$ \textit{(i)}. 
Let us suppose that for each $s \in S$, we have that $A^{s}$ is $s$-Koszul, $r$-$\pd_{A^{s}}(A) \leq 1$ and $\Ker(\delta_{2}) = \oplus_{s \in S} (\Ker(\delta_{2}) \cap (A \otimes R_{s}))$. 
Since $A^{s}$ is $s$-Koszul for each $s$, the Koszul complex $(K(A^{s})_{\bullet},\delta_{\bullet}^{s})$ is a minimal projective resolution of $k$ in the category of graded $A^{s}$-modules. 
Moreover, the condition $r$-$\pd_{A^{s}}(A) \leq 1$ together with the previously seen isomorphisms 
$(K(A)^{s}_{\bullet},\delta_{\bullet}|_{K(A)^{s}_{\bullet}}) \simeq A \otimes_{A^{s}} (K(A^{s})_{\bullet},\delta_{\bullet}^{s})$ 
imply that $K(A)^{s}_{\bullet}$ is acyclic in homological degrees greater than or equal to $2$, for each $s \in S$. 
Since $K(A)_{\bullet} = \oplus_{s \in S} K(A)_{\bullet}^{s}$ for $\bullet \geq 2$, and $\delta_{\bullet} = \oplus_{s \in S} \delta_{\bullet}|_{K(A)_{\bullet}^{s}}$, for $\bullet \geq 3$, 
we get that the multi-Koszul complex is acyclic in homological degrees greater than or equal to $3$. 
Finally, the last condition of \textit{(ii)} and the exactness of $K(A)_{\bullet}^{s}$ in homological degree $2$ for each $s \in S$ imply that 
\[     
\Ker(\delta_{2})  = \bigoplus_{s \in S} (\Ker(\delta_{2}) \cap (A \otimes R_{s})) = \bigoplus_{s \in S} (\mathrm{Im}(\delta_{3}) \cap (A \otimes R_{s})) = \im(\delta_{3}),     
\]
so $K(A)_{\bullet}$ is also acyclic in homological degree $2$. 
Again, since the multi-Koszul complex of any connected graded algebra is always acyclic in the first homological degree, 
we see that $(K(A)_{\bullet},\delta_{\bullet})$ is acyclic in positive homological degrees, 
so $A$ is multi-Koszul. 
The theorem is thus proved. 
\end{proof}

%%%%%%%
\begin{remark}
\label{rem:kosberg0}
Note that the algebras $A^{s}$ in the theorem may depend on the choice of the space of relations $R$. 
The abuse of language is however mild, meaning that the statement holds for any such choice. 
Moreover, we could have equivalently considered the projective dimension of the corresponding left module structure of $A$ over $A^{s}$, for each $s \in S$, 
and the decomposition property for the differential $\delta_{2}'$ of the right multi-Koszul complex of $A$.  
\end{remark}
%%%%%%% 

\begin{remark}
\label{rem:kosberg}
Theorem \ref{thm:koszulbypeaces} and \cite{B2}, Thm. 2.11, tell us that if $A$ is an $S$-multi-Koszul algebra, then, since each $A^{s}$ is $s$-Koszul, for $s \in S$, 
we have the \textit{distributivity} condition on the triples $(E^{s},F^{s},G^{s})$ and 
$(E^{s},F^{s},(G')^{s})$ for $n \geq n_s(j+2)$, where 
\begin{align*}
   E^{s} &= V^{(n-n_s(j))} \otimes J_{j}^{s}, \, F^{s} = \cl{R_{s}}_{n-n_s(j)} \otimes V^{(n_s(j))}, \, \text{for each $j \in \NN$},
   \\
   G^{s} &= V^{(n-n_s(j+2)+1)} \otimes \cl{R_{s}}_{2 s -2} \otimes V^{(n_s(j-2) + 1)}, \, \text{if $j$ is even}, 
   \\
   (G')^{s} &= V^{(n-n_s(j+1))} \otimes R_{s} \otimes V^{(n_s(j-1))}, \, \text{if $j$ is odd}, 
\end{align*}
and the so-called \textit{extra conditions}
\[     (V^{(s-1)} \otimes R_{s}) \cap (\sum_{j=0}^{s-2} V^{(j)} \otimes R_{s} \otimes V^{(s-1-j)}) \subseteq V^{(s-2)} \otimes J_{3}^{s}.     \]
We shall analyse this in more detail in Section \ref{sec:2grados}, for the special case of an algebra with relations in two degrees.
\end{remark}
%%%%%%%

Let $A = T(V)/\cl{R}$ be a multi-Koszul algebra.
As before, we also denote by $A^{s} = T(V)/\cl{R_{s}}$ the associated $s$-th homogeneous algebra, which is Koszul by the previous theorem, 
and we consider the canonical surjective morphisms of graded algebras $\pi_{s} : A^{s} \rightarrow A$, for $s \in S$. 
This allows us to consider any $A$-module as an $A^{s}$-module. 
We furthermore consider the surjective morphisms of graded algebras $\rho_{s} : T(V) \rightarrow A^{s}$. 

Let us denote $\iota_i^{s} : J_{i}^{s} \rightarrow J_i$ the canonical injection of vector spaces. 
We recall that the (left) multi-Koszul complex of $A$ is written as $K(A)_{\bullet} = (A \otimes J_{\bullet}, \delta_{\bullet})_{\bullet \geq 0}$, and analogously, 
the (left) Koszul complex of $A^{s}$ is written as $K(A^{s})_{\bullet} = (A^{s} \otimes J_{\bullet}^{s}, \delta_{\bullet}^{s})_{\bullet \geq 0}$. 
They are considered as augmented complexes if furthermore $K(A)_{-1} = k$ and $K(A^{s})_{-1} =k$, for all $s \in S$, respectively, and in fact $\delta_{0} : K(A)_{\bullet} \rightarrow k$ 
is a minimal projective resolution of $k$ in the category of graded (left) $A$-modules, and for each $s \in S$, $\delta_{0}^{s}: K(A^{s})_{\bullet} \rightarrow k$ 
is a minimal projective resolution of $k$ in the category of graded (left) $A^{s}$-modules.  

It is well-known that the Koszul complex of the tensor algebra $T(V)$ is of the form $K(T(V))_{0} = T(V)$, $K(T(V))_{1} = T(V) \otimes V$ and $K(T(V))_{\bullet} = 0$, if $\bullet \geq 2$. 
The morphism $\partial_{1} : K(T(V))_{1} \rightarrow K(T(V))_{0}$ is given by the restriction of the multiplication of $T(V)$. 
It is considered as an augmented complex if furthermore $K(T(V))_{-1} =k$, giving a minimal projective resolution in the category of graded $T(V)$-modules via 
$\partial_{0} : K(T(V))_{0} \rightarrow k$ defined by the augmentation of $T(V)$. 

In what follows, given a connected nonnegatively graded algebra $A$, $E(A) = \Ext^{\bullet}_{A}(k,k)$ will denote the associated \emph{Yoneda algebra}. 
It is also a connected nonnegatively graded algebra with the cohomological degree $\bullet$ and the Yoneda product. 

For each $s \in S$, the morphism $\pi_{s}$ induces in turn the following morphism $\Pi_{\bullet}^{s} : K(A^{s})_{\bullet} \rightarrow K(A)_{\bullet}$ of complexes 
of (left) $A^{s}$-modules defined by $\Pi_{\bullet}^{s} = \pi_{s} \otimes \iota_{\bullet}^{s}$. 
It is straightforward to check that it commutes with the differential, and in fact it commutes with the augmentation if we define $\Pi_{-1}^{s} = 1_{k}$. 
This in turn induces a morphism $p_{s} : \Hom_{A^{s}} (K(A)_{\bullet},k) \rightarrow \Hom_{A^{s}} (K(A^{s})_{\bullet},k)$ of complexes of vector spaces 
(the differentials are zero by minimality of the resolutions, (\textit{cf.} \cite{C}, Exp. 15, Prop. 10, or \cite{B4}, Prop. 2.4). 
Notice that there are canonical isomorphisms $\Hom_{A^{s}} (K(A)_{\bullet},k) = \Hom_{A} (K(A)_{\bullet},k) \simeq \Ext^{\bullet}_{A} (k,k)$ 
and also $\Hom_{A} (K(A)_{\bullet},k) \simeq J_{\bullet}^{*}$. 
Analogously, we have $(J_{\bullet}^{s})^{*} \simeq \Hom_{A^{s}} (K(A^{s})_{\bullet},k) \simeq \Ext^{\bullet}_{A^{s}} (k,k)$. 
In this fashion, we also write $p_{s} : \Ext^{\bullet}_{A} (k,k) \rightarrow \Ext^{\bullet}_{A^{s}} (k,k)$ for the corresponding morphism of graded vector spaces (with the cohomological degree). 
By using the previous identification, we see that $p_{s} = \oplus_{\bullet \in \NN_{0}} (\iota_{\bullet}^{s})^{*}$, so it is surjective.  

We claim that in fact the previous morphisms $p_{s}$ are also compatible with the Yoneda product, so they induce morphisms of graded algebras. 
This follows from the fact that the product for the Yoneda algebra of a connected graded algebra $B$ is induced by the (unique up to homotopy) 
morphism of augmented complexes $\Delta_{\bullet}(B) : K(B)_{\bullet} \rightarrow K(B)_{\bullet} \otimes K(B)_{\bullet}$ lifting the identity $k \rightarrow k \simeq k \otimes k$, 
for $K(B)_{\bullet}$ the augmented complex associated to a projective resolution of the (left) $B$-module $k$, 
by the formula $\alpha \cup \beta = (\alpha \otimes \beta) \circ \Delta_{\bullet}(B)$, for $\alpha, \beta$ homogeneous cocycles of $\Hom_{B}(K(B)_{\bullet},k)$ 
representing cohomology classes in $\Ext_{B}^{\bullet}(k,k)$. 
If the projective resolution $K(B)_{\bullet}$ is minimal, the identification $\Hom_{B}(K(B)_{\bullet},k) \simeq \Ext_{B}^{\bullet}(k,k)$ even simplifies the treatment. In the previous case, by the Comparison Theorem 2.2.6 and Porism 2.2.7 in \cite{W}, we may conclude that the following diagram 
\[
\xymatrix
{
K(A^{s})_{\bullet}
\ar[rd]^{\Delta_{\bullet}^{s}}
\ar[dd]^{\Pi_{\bullet}^{s}}
\ar[rr]
&
&
k
\ar@{=}[dd]
&
\\
&
K(A^{s})_{\bullet} \otimes K(A^{s})_{\bullet}
\ar[dd]^(0.4){\Pi_{\bullet}^{s} \otimes \Pi_{\bullet}^{s}}
\ar[ru]
&
&
\\
K(A)_{\bullet}
\ar[rd]^{\Delta_{\bullet}}
\ar[rr]
&
&
k
&
\\
&
K(A)_{\bullet} \otimes K(A)_{\bullet}
\ar[ru]
&
&
}     
\]
commutes up to (unique) homotopy, where $\Delta_{\bullet} = \Delta_{\bullet}(A)$ and $\Delta_{\bullet}^{s} = \Delta_{\bullet}(A^{s})$.
This immediately implies that $p_{s}$ is compatible with the Yoneda product. 

In the same manner, for each $s \in S$, the morphism $\rho_{s}$ induces a morphism $F_{\bullet}^{s} : K(T(V))_{\bullet} \rightarrow K(A^{s})_{\bullet}$ of complexes of (left) $A^{s}$-modules defined by $F_{\bullet}^{s} = \rho_{s} \otimes 1$, if $\bullet = 0, 1$, and zero for $\bullet \geq 2$. 
Again, it is clear that it commutes with the differential, and in fact it commutes with the augmentation if we define $F_{-1}^{s} = 1_{k}$. 
It induces a morphism $q_{s} : \Hom_{T(V)} (K(T(V))_{\bullet},k) \rightarrow \Hom_{T(V)} (K(A^{s})_{\bullet},k)$ of complexes of vector spaces 
(the differentials are zero by minimality of the resolutions). 
There are canonical isomorphisms $\Hom_{T(V)} (K(A^{s})_{\bullet},k) = \Hom_{A^{s}} (K(A^{s})_{\bullet},k) \simeq \Ext^{\bullet}_{A^{s}} (k,k) \simeq (J_{\bullet}^{s})^{*}$ and $\Hom_{T(V)} (K(T(V))_{\bullet},k) \simeq \Ext^{\bullet}_{T(V)} (k,k) \simeq k^{*} \oplus V^{*}$, where $V^{*}$ sits in cohomological degree $1$ and $k^{*}$ in cohomological degree zero. 
Hence, we may also write $q_{s} : \Ext^{\bullet}_{A^{s}} (k,k) \rightarrow \Ext^{\bullet}_{T(V)} (k,k)$ for the corresponding morphism of graded vector spaces (with the cohomological degree). 
By the previous identifications we see that $q_{s}$ is given by the projection on the cohomological degrees less than or equal to $1$, so it is surjective. 
By the same argument as before (or just regarding the explicit formulas for the Yoneda product given in Proposition 3.1 of \cite{BM}) we see that $q_{s}$ is a morphism of graded algebras. 

%%%%%%%
\begin{remark}
\label{rem:yoneda0}
Note that in fact the Yoneda algebra is a functorial construction, 
where the definition on the morphisms is as follows (we refer to \cite{LH}, specially Chapters 1 and 3). 
Given any morphism of (augmented) algebras $A \rightarrow A'$, 
it induces a morphism of differential graded coalgebras between their \textit{bar constructions} $B(A) \rightarrow B(A')$ 
(induced by the cofreeness property of $B(A)$ and the map of graded vector spaces $B(A) \rightarrow A'$ given by the composition of the canonical projection $B(A) \rightarrow A$ 
and of $A \rightarrow A'$), so by taking graded duals we obtain a morphism of differential graded algebras $B(A')^{\#} \rightarrow B(A)^{\#}$, 
which in turn induces a morphism of graded algebras $H^{\bullet}(B(A')^{\#}) \rightarrow H^{\bullet}(B(A)^{\#})$, 
\textit{i.e.} a morphism $E(A') \rightarrow E(A)$, since the Yoneda algebra of $A$ (resp., $A'$) coincides with $H^{\bullet}(B(A)^{\#})$ (resp., $H^{\bullet}(B(A')^{\#})$). 

In particular, the morphism $\pi_{s} : A^{s} \rightarrow A$ always induces a morphism of differential graded algebras $\Pi_{s} : B(A)^{\#} \rightarrow B(A^{s})^{\#}$, 
so a morphism of graded algebras $E(A) \rightarrow E(A^{s})$, and we may check that it trivially coincides with the morphism induced by the previously seen map $\Pi^{s}$, 
by constructing comparison morphisms from the minimal projective resolutions to the corresponding reduced bar complexes. 
The same applies to the morphism $\rho_{s} : T(V) \rightarrow A^{s}$, the induced morphism of differential graded algebras $F_{s} : B(A^{s})^{\#} \rightarrow B(T(V))^{\#}$ 
and the consequent map $E(A^{s}) \rightarrow E(T(V))$. 
Note also the commutativity $F_{s} \circ \Pi_{s} = F_{s'} \circ \Pi_{s'}$, for all $s, s' \in S$.

Even though we may have used these constructions to prove all the previous results, 
it is however useful to have the explicit form of the maps induced on the Yoneda algebras coming from the minimal projective resolutions in order to prove the next proposition. 
\end{remark}
%%%%%%%

By the explicit expression of the previous morphisms, it is trivial to check that $q_{s} \circ p_{s} = q_{s'} \circ p_{s'}$, for all $s, s' \in S$. 
Let us denote $E$ the inverse limit of the diagram in the category of graded algebras given by $\{ q_{s} : E(A^{s}) \rightarrow E(T(V)), \text{ for $s \in S$} \}$. 
By the universal property of $E$, and the aforementioned commutativity, there exists a unique morphism of graded algebras $\iota : E(A) \rightarrow E$. 
However, by the explicit expression of the morphisms $p_{s}$, for $s \in S$, it is easy to see that $E(A)$ is the inverse limit in the category of graded vector spaces of the system 
$\{ q_{s} : E(A^{s}) \rightarrow E(T(V)), \text{ for $s \in S$} \}$. 
Since the forgetful functor from the category of graded algebras to the category of graded vector spaces preserves inverse limits, 
$E$ is also an inverse limit in the category of graded vector spaces, so \textit{a fortiori} the morphism $\iota$ is bijective, and hence an isomorphism of graded algebras. 
This implies the next result.
%%%%%%%
\begin{proposition} 
\label{prop:yoneda}
The algebra $E(A)$ of a multi-Koszul algebra $A$ is the inverse limit in the category of graded algebras of the system given by 
$\{ q_{s} : E(A^{s}) \rightarrow E(T(V)), \text{ for $s \in S$} \}$. 
\end{proposition}
%%%%%%%

\begin{remark}
\label{rem:yoneda}
By the previous result and the explicit characterization of the algebra structure of the Yoneda algebra of a generalized Koszul algebra given in Proposition 3.1 of \cite{BM}, 
we easily get the Yoneda product of the algebra $E(A)$ for a multi-Koszul algebra.   
\end{remark}
%%%%%%%

Again, using the explicit expression of the morphisms $p_{s}$, and Theorem 4.1 in \cite{GMMZ}, we also obtain the following consequence. 
%%%%%%%
\begin{corollary} 
\label{coro:yoneda}
The graded algebra $E(A)$ of a multi-Koszul algebra $A$ is generated by $E^{1}(A) = \Ext_A^1 (k,k)$ and $E^{2}(A) = \Ext_A^2 (k,k)$, \textit{i.e.} it is $\mathcal{K}_{2}$ 
(in the sense of Cassidy and Shelton). 
\end{corollary} 
%%%%%%%

\begin{remark}
\label{rem:malo}
Contrary to what happens for homogeneous algebras (see \cite{GMMZ}, Thm. 4.1), the converse of the previous corollary does not hold in the case of multi-homogeneous algebras 
(\textit{e.g.} see the algebra $B$ in \cite{CG}, which is not $\{ 2, 3\}$-multi-Koszul). 
\end{remark}
%%%%%%%

\begin{remark}
\label{rem:yonedaainf}
We may also use the previous results to obtain a description of the $A_{\infty}$-algebra structure of the Yoneda algebra of a multi-Koszul algebra as follows (we also refer to \cite{LH}). 
By the commutative relation stated at the end of Remark \ref{rem:yoneda0} there exists a map of differential graded algebras from 
$B(A)^{*}$ to the inverse limit in the category of differential graded algebras of the system given by $\{ F_{s} : B(A^{s})^{*} \rightarrow B(T(V))^{*}, \text{ for $s \in S$} \}$. 
Moreover, by Proposition \ref{prop:yoneda}, it is in fact a quasi-isomorphism of differential graded algebras, so it induces a quasi-isomorphism of 
$A_{\infty}$-algebras from $E(A)$ to the $A_{\infty}$-algebra given by the cohomology of the inverse limit in the category of differential graded algebras of the system 
$\{ F_{s} : B(A^{s})^{*} \rightarrow B(T(V))^{*}, \text{ for $s \in S$} \}$. 
On the other hand, using the Merkulov's procedure to obtain a canonical $A_{\infty}$-algebra structure on the homology of a differential graded algebra (see \cite{Mer}), 
it is trivial to see that we may choose the higher multiplications of the inverse limit such that $m_{n} (a_{1} \otimes \dots \otimes a_{n}) = 0$ if there are indices $1 \leq i \neq j \leq n$ 
satisfying that $a_{i} \in E^{d_{i}}(A^{s})$ and $a_{j} \in E^{d_{j}}(A^{s'})$ for $s \neq s'$ and $d_{i}, d_{j} > 1$.
Furthermore, the restriction of the higher multiplication $m_{n}$ of $E(A)$ to each $E(A^{s})^{\otimes n}$ 
is given by the  corresponding Merkulov's construction of the higher multiplication of $E(A^{s})$, for each $s \in S$. 
By \cite{HeL}, Thm. 6.4 and 6.5, we get that the $A_{\infty}$-algebra $E(A)$ is generated in degree $1$, and we may in fact choose the higher multiplications 
such that $m_{n}$ vanishes if $n \notin S \cup \{ 2 \}$, and $m_{s} = m_{s}|_{E(A^{s})^{\otimes s}}$, for $s \in S \setminus \{ 2 \}$, 
is given by the usual expression for the case of generalized $s$-Koszul algebras given in the second of the mentioned theorems.  
\end{remark}
%%%%%%%

%%%%%%%%%%%%%%%%%%%%%%%%%%%%%%%%%%%%%%%%%%%%%%%%%%%%%%%%%%%%%%%%%%%%%%%%%%%%%%%%%%%%%%%%%%%%%%%%%%%%%%%%%%%%%%%%%%%%%%%%%%%%%%%%%%%%%%%%%%%%%%%%%%%%%%%%%

\section{\texorpdfstring{Multi-Koszul bimodule resolutions}{Multi-Koszul bimodule resolutions}}

In this section, we shall construct a minimal projective resolution of a multi-Koszul algebra $A$ in the category of (graded) $A$-bimodules, 
adapting the ideas of \cite{B2} (see also \cite{BM}), which is useful to compute its Hochschild (co)homology. 

As usual, we consider $A$ to be an $S$-multi-homogeneous algebra, for a finite subset $S \subseteq \NN_{\geq 2}$. 
We will denote the abelian category of $\mathbb{Z}$-graded left bounded $A$-bimodules with degree preserving bimodule morphisms by $\mathcal{C}$. 
Let $A^{e}=A\otimes_k A^{op}$ be the enveloping algebra of $A$. 
It is well-known that $\mathcal{C}$ is naturally isomorphic to the category of $\mathbb{Z}$-graded left bounded left $A^{e}$-modules, so we can use the results and notations of 
Section \ref{sec:preliminaries} for the graded algebra $A^e$. 

If $s\in S$, we will denote
\[     \bar{J}_{m}^{s} = \bigcap_{j=0}^{m-s} V^{(j)} \otimes R_{s} \otimes V^{(m-s-j)},     \]
if $m \geq s$, and $\bar{J}_{m}^{s} = V^{(m)}$, if $0 \le m < s$. 
Notice that $J_{i}^{s} = \bar{J}_{n_{s}(i)}^{s}$, for all $i \in \NN_{0}$ and $s \in S$. 

For $s\in S$ and $i \geq 0$ we consider the left $A$-module $\bar{K}_{L,s}(A)_i = A \otimes \bar{J}_{i}^{s}$ and the $A$-linear morphism 
$(\delta_{L,s})_i: \bar{K}_{L,s}(A)_i \rightarrow \bar{K}_{L,s}(A)_{i-1}$ induced by $\alpha \otimes v_{1} \dots v_{i} \rightarrow \alpha \cdot v_{1} \otimes v_{2} \dots v_{i}$. 
It is clear that $(\delta_{L,s})^s=0$ and then $(\bar{K}_{L,s}(A)_{\bullet},(\delta_{L,s})_{\bullet})$ is an $s$-complex. 
Analogously, we consider the $s$-complex $(\bar{K}_{R,s}(A)_{\bullet},(\delta_{R,s})_{\bullet})$, where $\bar{K}_{R,s}(A)_i= \bar{J}_{i}^s\otimes A$, 
and where the differential $(\delta_{R,s})_i: \bar{K}_{R,s}(A)_i \rightarrow \bar{K}_{R,s}(A)_{i-1}$ is induced by 
$v_{i} \dots v_{1} \otimes \alpha \rightarrow v_{i} \dots v_{2} \otimes v_{1} \cdot \alpha$. 
We define $\bar{K}_{L-R,s}(A)_{\bullet}$ as $\bar{K}_{L,s}(A)_{\bullet} \otimes A=A\otimes \bar{K}_{R,s}(A)_{\bullet}$, 
which can be regarded as an $s$-complex of bimodules either with differential $\delta'_{L,s}=\delta_{L,s} \otimes 1_A$ or $\delta '_{R,s}=1_A \otimes \delta_{R,s}$. 
Note also that $\delta'_{L,s}$ and $\delta'_{R,s}$ commute.

We now consider the complex of $A$-bimodules $(K_{L-R}(A)_{\bullet},(\delta_{L-R})_{\bullet})$ defined as follows. 
We set $K_{L-R}(A)_0=A\otimes A$, $K_{L-R}(A)_1=A\otimes V\otimes A$, 
and $K_{L-R}(A)_i = \bigoplus_{s\in S}\bar{K}_{L,s}(A)_{n_{s}(i)} \otimes A$, for $i \geq 2$, together with the differential $(\delta_{L-R})_{\bullet}$ 
defined as follows: $(\delta_{L-R})_{1}(\sum_j \alpha_j \otimes v_j \otimes \alpha'_j) = \sum_j (\alpha_j v_j \otimes \alpha'_j - \alpha_j \otimes v_j \alpha'_j)$, 
for $v_{j} \in V$ and $\alpha_j, \alpha'_j \in A$, and, for $i \geq 2$, we set $(\delta_{L-R})_{i} = \sum_{s\in S} (\delta_{L-R,s})_{i}$, where 
\[
(\delta_{L-R,s})_{i} =
\begin{cases}
\delta'_{L,s}-\delta'_{R,s}, & \text{if $i$ odd},
\\ 
\sum\limits_{j=0}^{s-1} (\delta'_{L,s})^{j} (\delta'_{R,s})^{(s-1-j)}, & \text{if $i$ even}.
\end{cases}
\]
It is clear that $(\delta_{L-R})_{i+1} \circ (\delta_{L-R})_{i}=0$, for $i \geq 0$, so $(K_{L-R}(A)_{\bullet},(\delta_{L-R})_{\bullet})$ is indeed a complex of pure projective $A$-bimodules. 
It can be regarded as an augmented complex for the augmentation $(\delta_{L-R})_{0} : K_{L-R}(A)_{0} \rightarrow A$ given by the product of the algebra. 
It is called the \emph{muti-Koszul bimodule complex} of $A$.
%%%%%%%
\begin{theorem}
Let $A$ be an $S$-multi-homogeneous algebra with space of relations $R = \oplus_{s \in S} R_s$, $R_s \subseteq V^{(s)}$, and $S$ a finite subset of $\NN_{\ge 2}$.
The augmented multi-Koszul bimodule complex 
\begin{equation}
\label{eq:bimodkoszulcomplex}
\cdots \longrightarrow K_{L-R}(A)_2 \overset{(\delta_{L-R})_{2}}{\longrightarrow} K_{L-R}(A)_1 \overset{(\delta_{L-R})_{1}}{\longrightarrow} K_{L-R}(A)_0 \overset{(\delta_{L-R})_{0}}{\longrightarrow} A \longrightarrow 0
\end{equation}
is exact if and only if $A$ is multi-Koszul.
\end{theorem}
%%%%%%%
\begin{proof}
We suppose that $A$ is multi-Koszul. 
Applying the functor $(\place) \otimes_A k$ to \eqref{eq:bimodkoszulcomplex}, we obtain the (augmented) complex $(K(A)_{\bullet},\delta_{\bullet})$, 
which is exact when $A$ is multi-Koszul. 
Since the $A$-bimodules $K_{L-R}(A)_i$ are graded-free and left bounded for all $i\in \mathbb{N}_0$,
Lemma \ref{lemma:nakacons} implies that the complex \eqref{eq:bimodkoszulcomplex} is exact.

Assume now that \eqref{eq:bimodkoszulcomplex} is exact. 
Let us denote by $\mathcal{C}_R$ the abelian category of $\mathbb{Z}$-graded left bounded right $A$-modules. 
Since \eqref{eq:bimodkoszulcomplex} is exact, it is a projective resolution of $A$ in $\mathcal{C}_R$, so, taking into account that $A$ is projective in $\mathcal{C}_R$, 
the complex \eqref{eq:bimodkoszulcomplex} is homotopically trivial as a complex of objects of graded right $A$-modules. 
Therefore, its image under the functor $(\place) \otimes_A k$ is \textit{a fortiori} homotopically trivial (as a complex of vector spaces). 
Since this image is the left multi-Koszul complex of $A$, it is exact in positive degrees, so $A$ is multi-Koszul. 
\end{proof}

%%%%%%%
\begin{remark}
If $A$ is multi-Koszul, since there is an obvious isomorphism of complexes of the form $k \otimes_{A^{e}} K_{L-R}(A)_{\bullet} \simeq k \otimes_{A} K(A)_{\bullet}$, 
having in fact vanishing differential, by the comments in the antepenultimate paragraph of Section \ref{sec:preliminaries} it follows that 
the complex \eqref{eq:bimodkoszulcomplex} is a minimal projective resolution of $A$ in $\mathcal{C}$.
\end{remark}
%%%%%%%

Using the previous complex, since the Hochschild homology groups $HH_{\bullet}(A)$ are isomorphic to $\Tor_{\bullet}^{A^{e}} (A,A)$ (for $k$ is a field),  
we see that $HH_{\bullet}(A)$ may be computed by $H_{\bullet} (A \otimes_{A^e} K_{L-R}(A)_{\bullet},1_A\otimes (\delta_{L-R})_{\bullet})$. 
In the same way, the Hochschild cohomology groups $HH^{\bullet}(A)$ are isomorphic to $\Ext^{\bullet}_{A^{e}} (A,A)$, so 
$HH^{\bullet}(A)$ can be computed from $H^{\bullet} (\Hom_{A^e} (K_{L-R}(A)_{\bullet},A), \Hom((\delta_{L-R})_{\bullet},1_{A}))$. 

%%%%%%%%%%%%%%%%%%%%%%%%%%%%%%%%%%%%%%%%%%%%%%%%%%%%%%%%%%%%%%%%%%%%%%%%%%%%%%%%%%%%%%%%%%%%%%%%%%%%%%%%%%%%%%%%%%%%%%%%%%%%%%%%%%%%%%%%%%%%%%%%%%%%%%%

\section{\texorpdfstring{Multi-Koszul algebras for relations in two degrees}{Multi-Koszul algebras for relations in two degrees}}
\label{sec:2grados}

In this section,  we concentrate on the special case of having relations in only two degrees. 
We will find necessary and sufficient conditions on some sequences of vector subspaces of the tensor powers of the base space $V$ to obtain the multi-Koszul property, 
in an analogous manner as the case of (generalized) Koszul algebras done by Berger. 
As expected, several of the subspaces analysed and the conditions involved were already considered by Berger, 
but new sequences of subspaces satisfying more involved conditions are introduced in order to obtain the multi-Koszul definition. 
This section is somehow long and rather tedious, which seems to be unavoidable to us, since it includes all detailed computations and analysis of the mentioned lattices, 
but their importance justifies in our opinion the extension. 

%%%%%%%%%%%%%%%%%%%%%%%%%%%%%%%%%%%%%%%%%%%%%%%%%%%%%%%%%%%%%%%%%%%%%%%%%%%%%%%%%%%%%%%%%%%%%%%%%%%%%%%%%%%%%%%%%%%%%%%%%%%%%%%%%%%%%%%%%%%%%%%%%%%%%%%%%

\subsection{\texorpdfstring{Preliminaries}{Preliminaries}}

We shall proceed as follows. 
After some preliminaries, we will analyse the construction procedure of a minimal graded projective resolution of the trivial $A$-module $k$. 
The description of the successive kernels of the differentials in this resolution will lead us to several ``naturally occurring'' sequences of vector subspaces of tensor powers of $V$, 
and we will impose necessary and sufficient conditions on these lattices to get the multi-Koszul property on the algebra $A$. 
The situation is summarized in Theorem \ref{thm:Kozulequiv}.
We advice the reader that we will use the results of Section 2, specially the ones concerning the construction of projective covers 
and minimal projective resolutions of bounded below graded modules. 

Next, we state an easy lemma concerning vector spaces that will be used in the sequel without much mention.
%%%%%%%
\begin{lemma}
\label{lem:trivialspace}
Let $V$ and $W$ be vector spaces such that $V$ is nontrivial. 
If $V \otimes W=0$, then $W=0$.
Also, if $W$ and $W'$ are subspaces of the vector spaces $V$ and $V'$, respectively, 
then $(W\otimes V')\cap (V\otimes W')=W\otimes W'$.
\end{lemma}
%%%%%%%

From now on, we fix two integers $a$ and $b$ such that $2 \leq a<b$, $R_a$ and $R_b$ subspaces of $V^{(a)}$ and $V^{(b)}$, respectively, and $R=R_a \oplus R_b$. 
We also assume that $R$ satisfies the minimality condition \eqref{eq:minimality} which implies that $k\otimes_{T(V)} \cl{R} \simeq R_a\oplus R_b$ as graded vector spaces.
The two-sided ideal $I=\cl{R}$ generated by $R$ in the tensor algebra $T(V)$ is given by $I=\bigoplus_{n\in \mathbb{Z}} I_n$, where:
\begin{align*}
I_n &= 0, & \hskip 0.2cm & \text{if} \hskip 0.2cm n<a,
\\
I_n &=\sum\limits_{j=0}^{n-a} V^{(j)} \otimes R_a \otimes V^{(n-a-j)}, & \hskip 0.2cm & \text{if} \hskip 0.2cm a \leq n<b,
\\
I_n &= \sum\limits_{j=0}^{n-a} V^{(j)} \otimes R_a \otimes V^{(n-a-j)} + \sum\limits_{h=0}^{n-b} V^{(h)} \otimes R_b \otimes V^{(n-b-h)}, & \hskip 0.2cm & \text{if} \hskip 0.2cm b \leq n.
\end{align*}

We recall the fact already stated at the end of Section \ref{sec:preliminaries} that for any nonnegatively graded connected algebra $A = T(V)/\cl{R}$, 
where $V \simeq A_{>0}/(A_{>0} \cdot A_{>0})$ is a vector space spanned by a minimal set of (homogeneous) generators of $A$, and $R$ is a space of relations of $A$, 
the minimal projective resolution of the trivial (left) $A$-module $k$ begins as 
\[
A \otimes R \overset{\delta_2}{\longrightarrow} A \otimes V \overset{\delta_1}{\longrightarrow} A \overset{\delta_0}{\longrightarrow} k \longrightarrow 0,
\]
where $\delta_2$ is induced by the usual map $\alpha \otimes v_{1} \dots v_{n} \mapsto \alpha v_{1} \dots v_{n-1} \otimes v_{n}$. 
As previously explained we will deal with the case where $V$ is concentrated in degree $1$ and where $R$ is concentrated in degrees $a$ and $b$. 

We also recall the definition of the following vector spaces, both concentrated in degree $n$,
\begin{align*}
\bar{J}_n^a & = \bigcap\limits_{j=0}^{n-a} V^{(j)}\otimes R_a \otimes V^{(n-a-j)}, & \text{for} \hskip 0.2cm n \geq a,
\\
\bar{J}_n^b & = \bigcap\limits_{j=0}^{n-b} V^{(j)}\otimes R_b \otimes V^{(n-b-j)}, & \text{for} \hskip 0.2cm n \geq b.
\end{align*}
Notice that $\bar{J}_m^b\cap (V^{(m-n)} \otimes \bar{J}_n^a) =0$, for all $m \geq n$, by the minimality condition on $R$. 

%%%%%%%%%%%%%%%%%%%%%%%%%%%%%%%%%%%%%%%%%%%%%%%%%%%%%%%%%%%%%%%%%%%%%%%%%%%%%%%%%%%%%%%%%%%%%%%%%%%%%%%%%%%%%%%%%%%%%%%%%%%%%%%%%%%%%%%%%%%%%%%%%%%%%%%%%

\subsection{\texorpdfstring{Description of $\Ker(\delta_2)$}{Description of ker(delta2)}}

It follows easily from the definition of the map $\delta_{2}$ that the $n$-th homogeneous component of its kernel is
\[     \Ker(\delta_{2})_{n} = \frac{(V^{(n-a)} \otimes R_{a} \oplus V^{(n-b)} \otimes R_{b}) \cap (I_{n-1} \otimes V)}
                                                 {I_{n-a} \otimes R_{a} \oplus I_{n-b} \otimes R_{b}},     \]
where the direct sums appearing in the right member are due to the minimality condition on $R_{a} \oplus R_{b}$. 
Then
\[
(\Ker(\delta_2))_n =0 \hskip 0.2cm \text{if} \hskip 0.2cm n\le a \hskip 0.2cm \text{and} \hskip 0.2cm (\Ker(\delta_2))_{a+1} = \bar{J}_{a+1}^a,
\]
so there is an injection $\bar{J}_{a+1}^{a} \hookrightarrow k \otimes_{A} \Ker(\delta_{2})$ of graded vector spaces, 
and thus the projective cover of $\Ker(\delta_2)$ should include $A\otimes \bar{J}_{a+1}^a$, and furthermore, by Lemma \ref{lemma:tresimp}, it should also include $A\otimes \bar{J}_{b+1}^b$. 
It may also include other $s$-pure modules. 
We shall now analyse this situation in more detail. 

It is straightforward to see that for $n=a+m$ with $2\le m\le \min \{ a-1, b-a\}$
\[     (\Ker(\delta_2))_n = (V^{(m)}\otimes R_a) \cap \big(\sum\limits_{j=0}^{m-1} V^{(j)}\otimes R_a\otimes V^{(m-j)}\big)\supseteq V^{(m-1)}\otimes \bar{J}_{a+1}^a.     \]
Then, for the same indices as before, we have that $(\Ker(\delta_2))_n =A_{m-1} \cdot \bar{J}_{a+1}^a$ if and only if 
\begin{equation}
\label{eq:econa}
(V^{(m)}\otimes R_a)\cap \big(\sum\limits_{j=0}^{m-1}V^{(j)}\otimes R_a\otimes V^{(m-j)}\big) = V^{(m-1)}\otimes \bar{J}_{a+1}^a,
\end{equation}
where as usual the dot denotes the action of the ring $A$ on $(\Ker(\delta_2))_{a+1}=\bar{J}_{a+1}^a$.

Note that if \eqref{eq:econa} holds for $m=a-1$, then for $2 \leq m \leq a-1$ and $2 \leq t \leq m-2$, we have that
\[     (V^{(m)}\otimes R_a)\cap \Big(\sum_{j=m-t}^{m-1} V^{(j)} \otimes R_a \otimes V^{(m-j)}\Big)\subseteq V^{(m-1)}\otimes \bar{J}_{a+1}^a.     \]
By Lemma \ref{lem:trivialspace}, we get \eqref{eq:econa} for $2 \leq m \leq a-1$.

For $n=a+t=b+h$ with $1\le h\le 2a-b-1$, $(\Ker(\delta_2))_n$ is equal to the image under the map $T(V)\otimes R\rightarrow A\otimes R$ of
\[
[(V^{(t)}\otimes R_a) \oplus (V^{(h)}\otimes R_b)]\cap \big(\sum\limits_{j=0}^{t-1} V^{(j)}\otimes R_a\otimes V^{(t-j)}
+\sum\limits_{j'=0}^{h-1} V^{(j')}\otimes R_b\otimes V^{(h-j')}\big),
\]
which contains $(V^{(t-1)}\otimes \bar{J}_{a+1}^a)\oplus (V^{(h-1)}\otimes \bar{J}_{b+1}^b)$. 
Notice again that the last sum is direct by the minimality condition on $R$ and Lemma \ref{lem:trivialspace}.

Consider now the following relations
\begin{equation}
\label{eq:ec}
\begin{split}
(V^{(l)}\otimes R_a)\cap \big(\sum\limits_{j=0}^{l-1} V^{(j)}\otimes R_a \otimes V^{(l-j)}\big) & = V^{(l-1)}\otimes \bar{J}_{a+1}^a,
\\
(V^{(h)}\otimes R_b)\cap \big(\sum\limits_{j=0}^{h-1} V^{(j)}\otimes R_b \otimes V^{(h-j)}\big) & = V^{(h-1)}\otimes \bar{J}_{b+1}^b.
\end{split}
\end{equation}
With similar arguments as before, the second equation of \eqref{eq:ec} holds for $2 \leq h \leq b-1$ whenever it holds for $h=b-1$. 
The relations \eqref{eq:ec} for $l=a-1$ and $h=b-1$ will be called \emph{extra conditions} and will be abbreviated by \textit{e.c.}
Note also that the equality in the previous identities is clearly seen to be equivalent to the inclusion of the left members inside the right ones. 

The following standard definition will be generalized afterwards (\textit{cf.} \cites{O}, Ch. I, \S 4, Axiom $\delta$).
%%%%%%%
\begin{definition}
Given $t \in \NN$, a sequence $(E, F_{1}, \dots, F_{t})$ of subspaces of a given vector space is said to be \emph{distributive} if
\[     E\cap \big(\sum_{j = 1}^{t} F_{j}\big) = \sum_{j = 1}^{t} (E \cap F_{j}).     \]
\end{definition}
%%%%%%%

Notice that the inclusion of the right member inside the left one always holds. 

%%%%%%%
\begin{proposition}
\label{prop:eciffdistributivity}
The \textit{e.c.} hold if and only if for $2\le m\le a-1$ and $2\le h\le b-1$ the triples
\begin{align*}
&(V^{(m)}\otimes R_a, R_a\otimes V^{(m)}, \sum\limits_{j=1}^{m-1} V^{(j)}\otimes R_a \otimes V^{(m-j)}),
\\
&(V^{(h)}\otimes R_b, R_b\otimes V^{(h)}, \sum\limits_{j=1}^{h-1} V^{(j)}\otimes R_b \otimes V^{(h-j)})
\end{align*}
are distributive and there are inclusions
\begin{align*}
(V^{(m)}\otimes R_a)\cap (R_a \otimes V^{(m)}) & \subseteq V^{(m-1)}\otimes R_a\otimes V,
\\
(V^{(h)}\otimes R_b)\cap (R_b \otimes V^{(h)}) & \subseteq V^{(h-1)}\otimes R_b \otimes V.
\end{align*}
\end{proposition}
%%%%%%%
\begin{proof}
The proof is the same (for each degree $a$ and $b$) as the one given for homogeneous algebras in \cite{B2}, Prop. 2.5. 
\end{proof}

%%%%%%%
\begin{lemma}
\label{lem:equivalenceforincl}
For $2\le m\le a-1$ and $2\le h\le b-1$, the inclusions 
\begin{align*}
(V^{(m)}\otimes R_a)\cap (R_a\otimes V^{(m)}) &\subseteq V^{(m-1)}\otimes R_a\otimes V,
\\
(V^{(h)}\otimes R_b)\cap (R_b\otimes V^{(h)}) &\subseteq V^{(h-1)}\otimes R_b\otimes V
\end{align*}
hold if and only if the equalities 
\[
(V^{(m)}\otimes R_a)\cap (R_a\otimes V^{(m)})=\bar{J}_{a+m}^a, \hskip 0.2cm (V^{(h)}\otimes R_b)\cap (R_b\otimes V^{(h)})=\bar{J}_{b+h}^b
\]
are satisfied. 
\end{lemma}
%%%%%%%
\begin{proof}
The proof is the same (for each degree $a$ and $b$) as the one given for homogeneous algebras in \cite{B2}, Lemma 2.6. 
\end{proof}

%%%%%%%
\begin{definition}
\label{def:multdisttriple}
Given $t,t'\in \mathbb{N}$, a tuple $(E, E',$ $ F_1,\cdots ,F_t, G_1, \cdots ,G_{t'})$ of subspaces of a given vector space is said \emph{$(t,t')$-bidistributive} if $E\cap E'=0$ and
\[
(E\oplus E')\cap \Big(\sum\limits_{j=1}^{t} F_j + \sum\limits_{j'=1}^{t'} G_{j'}\Big)
= \sum\limits_{j=1}^{t} (E\cap F_{j}) \oplus \sum\limits_{j'=1}^{t'} (E'\cap G_{j'}).
\]
The last identity may be replaced by the equivalent condition given by the inclusion of the left-hand member in the right one. 
Note further that $(E, E', F_1, \cdots , F_t, G_1, \cdots , G_{t'})$ is $(t,t')$-bidistributive 
if and only if $(E', E, G_1, \cdots ,$ $G_{t'}, F_1, \cdots , F_t)$ is $(t',t)$-bidistributive.
\end{definition}
%%%%%%%

We have the following simple characterization of the bidistributivity property.
%%%%%%%
\begin{lemma}
\label{lem:crossedincl}
Given a sequence $(E, E',F_1, \cdots , F_t, G_1, \cdots , G_{t'})$ of subspaces of a fixed vector space satisfying that $E \cap E' = 0$, it is $(t,t')$-bidistributive if and only if the triple 
$(\sum_{j=1}^{t} F_{j} + \sum_{j'=1}^{t'} G_{j'},E,E')$ and the sequences $(E, F_1, \cdots , F_{t}, G_1, \cdots , G_{t'})$ and $(E',F_{1}, \cdots , F_{t}, G_{1}, \cdots , G_{t'})$ are distributive, 
and there are inclusions
\[     \sum_{j'=1}^{t'} (E \cap G_{j'}) \subseteq \sum_{j=1}^{t} (E \cap F_j), 
\hskip 0.5cm
\sum_{j=1}^{t} (E' \cap F_j) \subseteq \sum_{j'=1}^{t'} (E' \cap G_{j'}).
\]
\end{lemma}
%%%%%%%
\begin{proof}
Suppose that $(E, E',F_1, \cdots , F_t, G_1, \cdots , G_{t'})$ is $(t,t')$-bidistributive.
By the obvious inclusion
\[     E \cap \big(\sum_{j=1}^{t} F_{j} + \sum_{j=1}^{t'} G_{j'}\big) \subseteq (E \oplus E') \cap (\sum_{j=1}^{t} F_{j} + \sum_{j=1}^{t'} G_{j'}),     \]
the bidistributivity property and the fact that $E \cap E' = 0$, we see that 
\[     E \cap \big(\sum_{j=1}^{t} F_{j} + \sum_{j=1}^{t'} G_{j'}\big) \subseteq \sum_{j=1}^{t} (E \cap F_{j}).     \]
Since the right member of the first inclusion is trivially included in $\sum_{j=1}^{t} (E \cap F_{j}) + \sum_{j'=1}^{t'} (E \cap G_{j'})$, we get the distributivity of 
the tuple $(E, F_1, \cdots , F_{t}, G_1, \cdots , G_{t'})$ and also the first of the stated inclusions. 
The distributivity of the other sequence $(E', F_1, \cdots , F_{t}, G_1, \cdots , G_{t'})$ and the second inclusion follow in the same way. 
It remains to prove that the triple $(\sum_{j=1}^{t} F_{j} + \sum_{j'=1}^{t'} G_{j'},E,E')$ is also distributive, which can be shown as follows. 
The $(t,t')$-bidistributive property tells us that 
\[     \big(\sum_{j=1}^{t} F_{j} + \sum_{j'=1}^{t'} G_{j'}\big) \cap (E + E') = \sum_{j = 1}^{t} (F_{j} \cap E) + \sum_{j' = 1}^{t'} (G_{j'} \cap E').     \]
By the inclusions stated in the lemma and the distributivity of the two already analysed sequences, we have that the right member of the previous identity 
coincides with 
\[     \big(\sum_{j=1}^{t} F_{j} + \sum_{j=1}^{t'} G_{j'}\big) \cap E + \big(\sum_{j=1}^{t} F_{j} + \sum_{j=1}^{t'} G_{j'}\big) \cap E',     \]
so the distributivity of the mentioned triple follows. 

The converse is immediate.  
\end{proof}

The following proposition gives an equivalence in terms of the bidistributivity and the \textit{e.c.}, 
to decide when $\Ker(\delta_2)$ is $2$-pure in degrees $a+1$ and $b+1$ satisfying that $k\otimes_A \Ker(\delta_2) \simeq \bar{J}_{a+1}^a\oplus \bar{J}_{b+1}^b$. 
The purpose will be then to generalize this equivalence for the other differentials in the resolution.
%%%%%%%
\begin{proposition}
\label{prop:kerdelta2pure}
The kernel $\Ker(\delta_2)$ is $2$-pure in degrees $a+1$ and $b+1$ satisfying that $k\otimes_A \Ker(\delta_2)\simeq \bar{J}_{a+1}^a\oplus \bar{J}_{b+1}^b$ 
if and only if the \textit{e.c.} 
hold and, for all $n \in \NN_{0}$ (or just for $n > a$), the tuple $(E',E'',F',G',F'',G'')$ is $(2,2)$-bidistributive, where:
\begin{align*}
E' &= V^{(n-a)}\otimes R_a, \,
G' = \sum\limits_{j=n-2a+1}^{n-a-1} V^{(j)}\otimes R_a\otimes V^{(n-a-j)},
\\
F' & = \sum\limits_{j=0}^{n-2a} V^{(j)}\otimes R_a \otimes V^{(n-a-j)} + \sum\limits_{j'=0}^{n-a-b} V^{(j')}\otimes R_b \otimes V^{(n-b-j')},
\\
E'' & = V^{(n-b)}\otimes R_b, \, G'' = \sum\limits_{j=n-2b+1}^{n-b-1} V^{(j)}\otimes R_b\otimes V^{(n-b-j)},
\\
F'' & = \sum\limits_{j=0}^{n-a-b} V^{(j)}\otimes R_a \otimes V^{(n-a-j)}+ \sum\limits_{j'=0}^{n-2b} V^{(j')}\otimes R_b \otimes V^{(n-b-j')}.
\end{align*}
\end{proposition}
%%%%%%%
\begin{proof}
We recall that $V^{(j)}=0$ for $j<0$. 
Notice also that, for $n < b$, the bidistributivity of $(E',E'',F',G',F'',G'')$ reduces to the distributivity of the triple 
$(E',F',G')$, where on each subspace all the summands with $R_{b}$ vanish by (tensor) degree reasons, 
giving thus similar expressions to the corresponding ones found in \cite{B2}, Prop. 2.7. 

We shall first prove the ``if'' part of the statement.  
Note that we have already showed that
\[
(\Ker(\delta_2))_n =A_{m-1} \cdot \bar{J}_{a+1}^a \hskip 0.2cm \text{for} \hskip 0.2cm n=m+a, \hskip 0.2cm \text{with} \hskip 0.2cm 2 \leq m \leq \min \{a-1,b-a\},
\]
if and only if \eqref{eq:econa} is satisfied. 
Fix an integer $n \geq \min\{a,b-a+1\}$, even though the proof also applies to arbitrary $n$, and suppose that the \textit{e.c.} and the bidistributivity condition on the previous tuple hold. 
We will use the explicit description of $(\Ker(\delta_2))_n$ for arbitrary $n$ given at the beginning of this subsection.

We have that $(\Ker(\delta_2))_n\subseteq A_{n-a}\otimes R_a\oplus A_{n-b}\otimes R_b$ and the subspace
\begin{align*}
N_n = & (V^{(n-a)}\otimes R_a \oplus V^{(n-b)}\otimes R_b)\cap 
\\
& \big( \sum\limits_{j=0}^{n-a-1} V^{(j)} \otimes R_a\otimes V^{(n-a-j)} + \sum\limits_{j'=0}^{n-b-1} 
 V^{(j')}\otimes R_b \otimes V^{(n-b-j')} \big)
\end{align*}
satisfies that
\[
(\Ker(\delta_2))_n = \frac{N_n}{I_{n-a}\otimes R_a \oplus I_{n-b}\otimes R_b}.
\]
Note that $(E'\oplus E'')\cap (F'+F''+G'+G'') =N_n$, and 
\begin{equation}
\begin{split}
\label{eq:saledeec}
E'\cap F' &= I_{n-a}\otimes R_a, \hskip 0.2cm E''\cap F'' = I_{n-b}\otimes R_b,
\\
E'\cap G' &= V^{(n-a-1)}\otimes \bar{J}_{a+1}^a, \hskip 0.2cm E''\cap G'' = V^{(n-b-1)}\otimes \bar{J}_{b+1}^b,
\end{split}
\end{equation}
where the first two equations are always satisfied and last two hold because of the \textit{e.c.} 
The kernel $(\Ker(\delta_2))_n$ is then given by 
\begin{align*}
& \frac{N_n}{I_{n-a}\otimes R_a \oplus I_{n-b} \otimes R_b} 
\\
& = \frac{(I_{n-a}\otimes R_a + V^{(n-a-1)}\otimes \bar{J}_{a+1}^a) \oplus (I_{n-b}\otimes R_b + V^{(n-b-1)}\otimes \bar{J}_{b+1}^b)}{I_{n-a}\otimes R_a\oplus I_{n-b}\otimes R_b} 
\\
& \simeq \frac{I_{n-a}\otimes R_a + V^{(n-a-1)}\otimes \bar{J}_{a+1}^a}{I_{n-a}\otimes R_a} \oplus \frac{I_{n-b}\otimes R_b + V^{(n-b-1)}\otimes \bar{J}_{b+1}^b}{I_{n-b}\otimes R_b}  
\\
& \simeq \frac{V^{(n-a-1)}\otimes \bar{J}_{a+1}^a}{(V^{(n-a-1)}\otimes \bar{J}_{a+1}^a)\cap (I_{n-a}\otimes R_a)}\oplus \frac{V^{(n-b-1)} \otimes \bar{J}_{b+1}^b}{(V^{(n-b-1)} \otimes \bar{J}_{b+1}^b) \cap (I_{n-b}\otimes R_b)}, 
\end{align*}
which is an epimorphic image of 
\[     \frac{V^{(n-a-1)}\otimes \bar{J}_{a+1}^a}{I_{n-a-1}\otimes \bar{J}_{a+1}^a} \oplus \frac{V^{(n-b-1)}\otimes \bar{J}_{b+1}^b}{I_{n-b-1}\otimes \bar{J}_{b+1}^b}\simeq A_{n-a-1} \otimes \bar{J}_{a+1}^a \oplus A_{n-b-1} \otimes \bar{J}_{b+1}^b.     \]
This follows from the obvious inclusion $I_{n-s-1}\otimes \bar{J}_{s+1}^s \subseteq (V^{(n-s-1)} \otimes \bar{J}_{s+1}^s) \cap (I_{n-s}\otimes R_s)$, for $s=a,b$.
We then conclude that $\Ker(\delta_2)$ is $2$-pure in degrees $a+1$ and $b+1$, satisfying that $k\otimes_A\Ker(\delta_2)\simeq \bar{J}_{a+1}^a\oplus \bar{J}_{b+1}^b$.

Conversely, we assume now that $\Ker(\delta_2)$ is $2$-pure in degrees $a+1$ and $b+1$ such that $k\otimes_A\Ker(\delta_2) \simeq \bar{J}_{a+1}^a\oplus \bar{J}_{b+1}^b$, 
so $(\Ker(\delta_2))_n= A_{n-a-1} \cdot \bar{J}_{a+1}^a \oplus A_{n-b-1} \cdot \bar{J}_{b+1}^b$. 
This implies the existence of the following exact sequence  
\[     A \otimes \bar{J}_{s+1}^{s} \rightarrow A \otimes R_{s} \rightarrow A \otimes V,     \]
for $s = a, b$, where the last map is induced by $\delta_{2}$, and the first one is easily seen to be induced by the differential $\delta_{3}$ 
given in the Definition \ref{def:multikoszul} of the multi-Koszul complex of $A$. 
By Lemma \ref{lemma:nakacons}, this implies that the following sequence 
\[     T(V)/\cl{R_{s}} \otimes \bar{J}_{s+1}^{s} \overset{\delta_{3}^{s}}{\rightarrow} T(V)/\cl{R_{s}} \otimes R_{s} \overset{\delta_{2}^{s}}{\rightarrow}  T(V)/\cl{R_{s}} \otimes V     \]
is also exact. 
We note that the morphisms of the previous sequence are those of the Koszul complex of $T(V)/\cl{R_{s}}$. 
This implies that $\Ker(\delta_{2}^{s})$ is pure in degree $s+1$, so, by \cite{B2}, Prop. 2.7, the \textit{e.c.} corresponding to $s = a, b$ holds.
 
We will now prove the bidistributive condition of the tuples in the statement. 
In order to do so, consider the following commutative diagram
\[
\xymatrix@C-50pt
{
\displaystyle{\bigoplus\limits_{s=a,b} \frac{V^{(n-s-1)}\otimes \bar{J}_{s+1}^s}{I_{n-s-1}\otimes \bar{J}_{s+1}^s}}
\ar@{->>}[rr]
\ar@{->>}[rdd]^{g}
& 
&
\displaystyle{\bigoplus\limits_{s=a,b} \frac{V^{(n-s-1)}\otimes \bar{J}_{s+1}^s+I_{n-s}\otimes R_s}{I_{n-s}\otimes R_s}}
\ar@{_{(}->}[ldd]^{f}
\\
&
&
\\
&
\displaystyle{\frac{(\bigoplus\limits_{s=a,b} V^{(n-s)}\otimes R_s)\cap (I_{n-1}\otimes V)}{\bigoplus\limits_{s=a,b} I_{n-s}\otimes R_s}}
& 
}
\]
The domain of $g$ is the homogeneous component of degree $n$ of $A \otimes (\bar{J}_{a+1}^{a} \oplus \bar{J}_{b+1}^{b})$, its codomain is 
the component of degree $n$ of $\Ker(\delta_{2})$ and $g$ is the homogeneous component of the projective cover, so surjective by the hypotheses. 
The horizontal epimorphism follows from the trivial inclusions $I_{n-s-1}\otimes \bar{J}_{s+1}^s \subseteq (V^{(n-s-1)}\otimes \bar{J}_{s+1}^s)\cap (I_{n-s}\otimes R_s)$ 
for $s=a, b$, and the map $f$ is just the canonical inclusion. 
The commutativity of the diagram just follows from the construction of the projective cover of a module. 
The surjectivity of $g$ yields the surjectivity of $f$, so it gives an equality between the domain of $f$ and its codomain. 
By the \textit{e.c.}, we have that the last two identities of \eqref{eq:saledeec} hold, which together with the equality coming from $f$, 
imply the $(2,2)$-bidistributivity of the respective tuple. 
The proposition is thus proved. 
\end{proof}

%%%%%%%
\begin{remark}
Notice that the spaces of the tuple considered in the previous proposition depend on $n$. 
We will however omit the index to simplify the notation, as in \cite{B2}. 
\end{remark}
%%%%%%%

\begin{remark}
\label{rem:isoind2}
We would like to stress that, contrary to what happens in Proposition 2.7 in \cite{B2} for the case of homogeneous algebras, 
the hypothesis $k \otimes_{A} \Ker(\delta_{2}) \simeq \bar{J}_{a+1}^{a} \oplus \bar{J}_{b+1}^{b}$ cannot be replaced by the weaker condition 
$\Ker(\delta_{2})$ is $2$-pure in degrees $a+1$ and $b+1$ (see the algebra considered in Example \ref{ex:difkos}). 

On the other hand, even though we may understand the condition $k \otimes_{A} \Ker(\delta_{2}) \simeq \bar{J}_{a+1}^{a} \oplus \bar{J}_{b+1}^{b}$ 
as stating that the obvious morphism from $\bar{J}_{a+1}^{a} \oplus \bar{J}_{b+1}^{b}$ to $k \otimes_{A} \Ker(\delta_{2})$ is an isomorphism, 
Lemma \ref{lemma:tresimp} tells us that we may consider that it states the existence of any isomorphism between both spaces, for the vector spaces are finite dimensional. 
\end{remark}
%%%%%%%

\begin{remark}
\label{rem:incimp}
By Lemma \ref{lem:crossedincl}, the bidistributivity condition on the previous proposition implies the following inclusions 
\begin{small}
\begin{align*}
    (V^{(m_{a})} \otimes R_{a}) \cap \Big(\sum_{j=0}^{m_{a}+a-b-1} V^{(j)} \otimes R_{b} \otimes V^{(m_{a}+a-b-j)}\Big) &\subseteq V^{(m_{a}-1)} \otimes \bar{J}_{a+1}^{a} + I_{m_{a}} \otimes R_{a},
    \\
   (V^{(m_{b})} \otimes R_{b}) \cap \Big(\sum_{j=0}^{m_{b}+b-a-1} V^{(j)} \otimes R_{a} \otimes V^{(m_{b}+b-a-j)}\Big) &\subseteq V^{(m_{b}-1)} \otimes \bar{J}_{b+1}^{b},
\end{align*}
\end{small}
for all $m_{a}, m_{b} \in \NN_{0}$ satisfying that $b-a+1 \leq m_{a} \leq b-1$, and $1 \leq m_{b} \leq a - 1$. 
Notice that the last inclusion would seem more symmetric if we had also written the subspace $I_{m_{b}} \otimes R_{a}$, which vanishes because $I_{m_{b}} = 0$. 
Furthermore, it is easily seen that the previous inclusions in fact imply that 
\begin{align*}
E' \cap F'' \subseteq E' \cap F' + E' \cap G', \hskip 0.5cm  &  E' \cap G'' \subseteq E' \cap F' + E' \cap G',
\\
E'' \cap F' \subseteq E'' \cap F'' + E'' \cap G'', \hskip 0.5cm  &  E'' \cap G' \subseteq E'' \cap F'' + E'' \cap G'',
\end{align*}
for the vector spaces $E'$, $E''$, $F'$, $F''$, $G'$ and $G''$ defined in the previous proposition. 
\end{remark}
%%%%%%%

We end this subsection by stating a result which is analogous to the one existing for homogeneous algebras.
%%%%%%%
\begin{proposition}
The \textit{e.c.} hold for $A^{\circ}$ if and only if $A$ satisfies its \textit{e.c.} 
\end{proposition}
%%%%%%%
\begin{proof}
The proof is the same (for each degree $a$ and $b$) as the one given for homogeneous algebras in \cite{B2}, Prop. 4.4.  
\end{proof}

%%%%%%%%%%%%%%%%%%%%%%%%%%%%%%%%%%%%%%%%%%%%%%%%%%%%%%%%%%%%%%%%%%%%%%%%%%%%%%%%%%%%%%%%%%%%%%%%%%%%%%%%%%%%%%%%%%%%%%%%%%%%%%%%%%%%%%%%%%%%%%%%%%%%%%%%

\subsection{\texorpdfstring{Description of $\Ker(\delta_i)$ for $i>2$}{Description of ker(deltai) for i>2}}
\label{sec:kerdeltaimay2}

From now on we assume that $\Ker(\delta_2)$ is $2$-pure in degrees $a+1$ and $b+1$, with $k \otimes_{A} \Ker(\delta_2) \simeq \bar{J}_{a+1}^{a} \oplus \bar{J}_{b+1}^{b}$. 
Moreover, consider as before $\delta_0$ to be the augmentation of the algebra and $\delta_1:A\otimes V\rightarrow A$ to be the restriction of the product of the algebra. 
Given $i \geq 3$, suppose that $\delta_2, \cdots ,\delta_{i-1}$ have been defined in such a way that the canonical injections
\[     \tilde{g}_j : \bar{J}_{n_a(j)}^a\oplus \bar{J}_{n_b(j)}^b \rightarrow \Ker(\delta_{j-1}) \subseteq A \otimes (\bar{J}_{n_a(j-1)}^a\oplus \bar{J}_{n_b(j-1)}^b),\text{ for all } 2\le j\le i,     \]
induce essential surjections $g_j:A\otimes (\bar{J}_{n_a(j)}^a \oplus \bar{J}_{n_b(j)}^b) \rightarrow \Ker(\delta_{j-1})$, whose coextensions 
$A\otimes (\bar{J}_{n_a(j)}^a \oplus \bar{J}_{n_b(j)}^b) \rightarrow A\otimes (\bar{J}_{n_a(j-1)}^a \oplus \bar{J}_{n_b(j-1)}^b)$ are the differentials $\delta_j$. 
Then $\delta_i : A \otimes (\bar{J}_{n_a(i)}^a \oplus \bar{J}_{n_b(i)}^b) \rightarrow A\otimes (\bar{J}_{n_a(i-1)}^a \oplus \bar{J}_{n_b(i-1)}^b)$ 
may be analogously defined as the coextension of $g_i$. 
Furthermore, the recursive procedure yields that, if $i > 2$, 
\[
\delta_i : A\otimes (\bar{J}_{n_a(i)}^a \oplus \bar{J}_{n_b(i)}^b) \longrightarrow A\otimes (\bar{J}_{n_a(i-1)}^a \oplus \bar{J}_{n_b(i-1)}^b)
\]
is a direct sum of two components $\delta_i^{a} \oplus \delta_i^{b}$, where $\delta_i^{s} : A \otimes \bar{J}_{n_s(i)}^s \rightarrow A \otimes \bar{J}_{n_s(i-1)}^s$, for $s=a,b$, 
because each of the morphisms $\tilde{g}_{i}$ satisfies the same property, is induced by the map  
\[     
\alpha \otimes v_{j_1}\cdots v_{j_{n_s(i)}} \mapsto \begin{cases}
\alpha v_{j_1}\cdots v_{j_{s-1}} \otimes v_{j_s}\cdots v_{j_{n_s(i)}}, & \text{if $i$ is even,}
\\
\alpha v_{j_1} \otimes v_{j_2}\cdots v_{j_{n_s(i)}}, & \text{if $i$ is odd.}
\end{cases}
\]

%%%%%%%
\begin{remark}
\label{rmk:gooddefinition}
Note that these morphisms are well-defined even considering their domains to be $V^{(n-n_s(i))}\otimes \bar{J}_{n_s(i)}^s$, for $s=a, b$.
\end{remark}
%%%%%%%

By the previous decomposition of the morphism $\delta_{i}$, we have that $\Ker(\delta_i) = \bigoplus_{s=a,b} (\Ker(\delta_i) \cap (A \otimes \bar{J}_{n_s(i)}^{s}))$, 
where $\Ker(\delta_i) \cap (A \otimes \bar{J}_{n_s(i)}^{s}) = \Ker(\delta_i^{s})$, so we may analyse each direct summand separately. 
Moreover, from the expression of $\delta_{i}^{s}$ we see that 
\[
(\Ker(\delta_i^{s}))_n 
= \frac{(V^{(n-n_{s}(i))} \otimes \bar{J}_{n_{s}(i)}^s) \cap (I_{n-n_{s}(i-1)}\otimes V^{(n_{s}(i-1))})}{I_{n-n_{s}(i)}\otimes \bar{J}_{n_{s}(i)}^s},
\]
for $n \in \NN_{0}$. 
On the other hand, we may rewrite $I_{n-n_{s}(i-1)}\otimes V^{(n_{s}(i-1))}$ as 
\[     I_{n-n_{s}(i)}\otimes V^{(n_{s}(i))} + \sum_{s'=a,b} \sum_{j = n -n_{s}(i)- s'+1}^{n-n_{s}(i-1) - s'} V^{(j)} \otimes R_{s'} \otimes V^{(n-s'-j)},     \]
where the first term will be denoted by $F^{s}$, and the summand corresponding to the index $s'$ of the second term will be denoted by $G^{s}_{s'}$. 
If we write $E^{s}$ instead of $V^{(n-n_{s}(i))} \otimes \bar{J}_{n_{s}(i)}^s$, the numerator of the previous quotient is given by $E^{s} \cap (F^{s} + G^{s}_{a} + G^{s}_{b})$, for $s = a, b$. 
Notice that the spaces considered in the previous paragraph depend on $n$ and $i$. 
We will however omit the indices to simplify the notation. 

We remark that the decomposition of the morphism $\delta_{i}$ immediately implies that the projective cover of $\Ker(\delta_{i})$ is the direct sum of the projective covers 
of each direct summand $\Ker(\delta_{i}^{s})$ for $s=a, b$. 
Note that $\Ker(\delta_{j}^{a})_{n} = 0$, if $n < n_{a}(j+1)$, and that $\Ker(\delta_{j}^{a})_{n_{a}(j+1)} = \bar{J}_{n_{a}(j+1)}^{a}$, for $j \leq i$, by Lemma \ref{lem:equivalenceforincl}, 
which is a consequence of the \textit{e.c.} 
This immediately implies that there is an injection of $\bar{J}_{n_{a}(j+1)}^{a}$ inside $(k \otimes_{A} \Ker(\delta_{j}))_{n_{a}(j+1)}$, given by the composition of the previous injection 
$\bar{J}_{n_{a}(j+1)}^{a} \hookrightarrow (\Ker(\delta_{j}))_{n_{a}(j+1)}$ and the canonical projection $\Ker(\delta_{j}) \rightarrow k \otimes_{A} \Ker(\delta_{j})$, for $j \leq i$, 
so the projective cover of $\Ker(\delta_{i})$ must include $A \otimes \bar{J}_{n_{a}(i+1)}^{a}$. 
On the other hand, by Lemma \ref{lemma:tresimp}, the composition of the inclusion $\bar{J}_{n_{b}(i+1)}^{b} \hookrightarrow \Ker(\delta_{i}^{b})$ with the canonical projection $\Ker(\delta_{i}^{b}) \rightarrow k \otimes_{A} \Ker(\delta_{i}^{b})$ is injective if and only if the homogeneous components $(k \otimes_{A} \Ker(\delta_{i}^{b}))_{n}$ vanish for $n < n_{b}(i+1)$. 

We remark that the equation 
\begin{equation}
\label{eq:prigen}
     E^{s} \cap F^{s} = I_{n-n_{s}(i)} \otimes \bar{J}_{n_{s}(i)}^s
\end{equation} 
always holds. 

Moreover, the \textit{e.c.} imply that 
\begin{equation}
\label{eq:seggen}
     E^{s} \cap G^{s}_{s} = V^{(n-n_{s}(i+1))} \otimes \bar{J}_{n_{s}(i+1)}^s.
\end{equation}  
The proof of this fact is somehow implicit in \cite{B2}, so we give a detailed proof just for convenience. 
It suffices to prove the inclusion of the left member inside the right one, for the other inclusion is direct. 
Moreover, note that $G^{s}_{s}$ vanishes for $n < n_{s}(i+1)$, by tensor degree reasons, and the same happens for the right member of \eqref{eq:seggen}, 
so it suffices to prove the identity (or the inclusion) only for $n \geq n_{s}(i+1)$. 

We start by writing the following trivial identity  
\[     E^{s} \cap G^{s}_{s} = E^s \cap (V^{(n-n_{s}(i))} \otimes R_{s} \otimes V^{(n_{s}(i-2))}) \cap G^{s}_{s}.     \]
We may consider two cases: $i$ odd or $i$ even. 
If $i$ is odd, then the sum in $G_{s}^{s}$ consists of only one summand, namely $V^{(n-n_{s}(i)-s+1)} \otimes R_{s} \otimes V^{(n_{s}(i)-1)}$, for $n_{s}(i-1) = n_{s}(i) - 1$. 
In this case we have that $n-n_{s}(i)-s+1 \geq 0$, because $n \geq n_{s}(i+1)$ and $i$ is odd. 
Hence, 
\begin{align*}
     E^{s} \cap G^{s}_{s} &= E^s \cap (V^{(n-n_{s}(i))} \otimes R_{s} \otimes V^{(n_{s}(i-2))}) \cap G^{s}_{s} 
\\
                                    &= E^{s} \cap \Big(V^{(n-n_{s}(i)-s+1)} \otimes \big((V^{(s-1)} \otimes R_{s}) \cap (R_{s} \otimes V^{(s-1)})\big) \otimes V^{(n_{s}(i-2))}\Big)
\\
                                    &=E^{s} \cap \Big(V^{(n-n_{s}(i)-s+1)} \otimes \bar{J}_{2s-1}^{s} \otimes V^{(n_{s}(i-2))}\Big) 
                                      = V^{(n-n_{s}(i+1))} \otimes \bar{J}_{n_{s}(i+1)}^s,
\end{align*}
where we have used in the penultimate equality the identities of Lemma \ref{lem:equivalenceforincl}, which hold due to the \textit{e.c.} and Proposition \ref{prop:eciffdistributivity}. 

We consider now the case $i$ is even. 
We point out that in this case $n-n_{s}(i)-s+1$ may be negative, so the sum is more tedious to handle. 
In fact, we get that 
\[     E^{s} \cap G^{s}_{s} = E^s \cap (V^{(n-n_{s}(i))} \otimes R_{s} \otimes V^{(n_{s}(i-2))}) \cap G^{s}_{s}      \]
can be rewritten as 
\[     E^{s} \cap \Big(V^{(\mathrm{Max})} \otimes \big((V^{(\mathrm{min})} \otimes R_{s}) \cap 
(\sum_{j=0}^{\mathrm{min}-1} V^{(j)} \otimes R_{s} \otimes V^{(\mathrm{min}-j)})\big) \otimes V^{(n_{s}(i-2))}\Big),     \]
where $\mathrm{Max} = \mathrm{max}\{0,n-n_{s}(i)-s+1\}$ and $\mathrm{min} = \mathrm{min}\{n-n_{s}(i),s-1\}$. 
By the \textit{e.c.}, it further coincides with 
\[     E^{s} \cap \Big(V^{(\mathrm{Max})} \otimes (V^{(\mathrm{min}-1)} \otimes \bar{J}_{s+1}^{s}) \otimes V^{(n_{s}(i-2))}\Big) = V^{(n-n_{s}(i+1))} \otimes \bar{J}_{n_{s}(i+1)}^s,     \] 
and the claim is proved.  

Finally, for a couple $(s,s')$ of different elements in $\{ a,b \}$, and $\mathrm{max}\{0, s'-s+1\} \leq m \leq s'-1$, we define the following intersection 
\begin{equation}
\label{eq:comoec}
    (V^{(m)} \otimes R_{s}) \cap \Big(\sum_{j=0}^{m+s-s'-1} V^{(j)} \otimes R_{s'} \otimes V^{(m+s-s'-j)}\Big),
\end{equation}
that will be denoted by $X^{s,m}_{s'}$. 

We claim that we have the following inclusion  
\begin{equation}
\label{eq:tergen}
     E^{s} \cap G^{s}_{s'} \subseteq E^{s} \cap \big(V^{(\mathrm{Max}')} \otimes X^{s,\mathrm{min}'}_{s'} \otimes V^{(n_{s}(i-2))}\big),
\end{equation}  
which is in fact an equality for $i$ even, 
where $\mathrm{Max}' = \mathrm{max}\{0,n-n_{s}(i)-s'+1\}$ and $\mathrm{min}' = \mathrm{min}\{n-n_{s}(i),s'-1\}$. 
If $i$ is odd the corresponding equality should be 
\begin{equation}
\label{eq:tergeniodd}
     E^{s} \cap G^{s}_{s'} = E^{s} \cap \Big(V^{(n-n_{s}(i)-s'+1)} \otimes \big((V^{(s'-1)} \otimes R_{s}) \cap (R_{s'} \otimes V^{(s-1)})\big) \otimes V^{(n_{s}(i-2))}\Big).
\end{equation}  
The proof is parallel to the previous one, but we sketch it for completeness. 
We remark that $G^{s}_{s'}$ vanishes for $n < n_{s}(i-1)+s'$, by tensor degree reasons, so it suffices to prove the inclusion (and the equality if $i$ is even) only for $n \geq n_{s}(i-1)+s'$. 

The first step is to write 
\[     E^{s} \cap G^{s}_{s'} = E^s \cap (V^{(n-n_{s}(i))} \otimes R_{s} \otimes V^{(n_{s}(i-2))}) \cap G^{s}_{s'}.     \]
We now consider two cases: $i$ odd or $i$ even. 
If $i$ is odd, then the sum in $G^{s}_{s'}$ consists of only one summand, namely $V^{(n-n_{s}(i)-s'+1)} \otimes R_{s'} \otimes V^{(n_{s}(i)-1)}$, for $n_{s}(i-1) = n_{s}(i) - 1$. 
In this case we also have $n-n_{s}(i)-s'+1 \geq 0$, because $n \geq n_{s}(i-1)+s'$ and $i$ is odd. 
So, $E^{s} \cap G^{s}_{s'}$ can be further rewritten as  
\[     E^{s} \cap \Big(V^{(n-n_{s}(i)-s'+1)} \otimes \big((V^{(s'-1)} \otimes R_{s}) \cap (R_{s'} \otimes V^{(s-1)})\big) \otimes V^{(n_{s}(i-2))}\Big),     \]
and the equality in \eqref{eq:tergeniodd} follows. 
Moreover, it is trivially included in 
\[     E^{s} \cap \Big(V^{(n-n_{s}(i)-s'+1)} \otimes \big((V^{(s'-1)} \otimes R_{s}) \cap (\sum_{j=0}^{s-2} V^{(j)} \otimes R_{s'} \otimes V^{(s-j-1)})\big) \otimes V^{(n_{s}(i-2))}\Big),      \]
which coincides with $E^{s} \cap (V^{(n-n_{s}(i)-s'+1)} \otimes X^{s,s'-1}_{s'} \otimes V^{(n_{s}(i-2))})$,  so the desired inclusion follows. 

We consider now the case $i$ is even, and we remark that in this case $n-n_{s}(i)-s'+1$ may be negative. 
In fact, we have that $E^{s} \cap G^{s}_{s'}$ can be rewritten as the intersection of $E^{s}$ with
\[     V^{(\mathrm{Max}')} \otimes \Big((V^{(\mathrm{min}')} \otimes R_{s}) \cap 
\Big(\sum_{j=0}^{\mathrm{min}'+s-s'-1} V^{(j)} \otimes R_{s} \otimes V^{(\mathrm{min}'+s-s'-j)}\Big)\Big) \otimes V^{(n_{s}(i-2))}.     \]
Hence, by the definition of $X^{s,m}_{s'}$ given in \eqref{eq:comoec}, the equality for $i$ even follows and the claim is proved.

In fact, for $s = b$ and $s' = a$, using the second inclusion of Remark \ref{rem:incimp} 
(which follows from the assumption that $\Ker(\delta_{2})$ satisfies that $k \otimes_{A} \Ker(\delta_{2}) \simeq \bar{J}_{a+1}^{a} \oplus \bar{J}_{b+1}^{b}$),  
the \textit{e.c.} and the definition of $\bar{J}_{n_{b}(i+1)}^{b}$, we have that 
\begin{equation}
\label{eq:tergenb}
     E^{b} \cap G^{b}_{a} \subseteq V^{(n_{b}(i+1))} \otimes \bar{J}_{n_{b}(i+1)}^{b}.
\end{equation}  

We state the following proposition which generalizes the result for $i = 2$. 
%%%%%%%
\begin{proposition}
\label{prop:kerdeltai}
Suppose that for all $j$ such that $2 \leq j<i$, $\Ker(\delta_j)$ is $2$-pure in degrees $n_a(j+1)$ and $n_b(j+1)$ satisfying that that 
$k\otimes_A \Ker(\delta_j) \simeq \bar{J}_{n_a(j+1)}^a\oplus \bar{J}_{n_b(j+1)}^b$. 
Then, $\Ker(\delta_i)$ is $2$-pure in degrees $n_a(i+1)$ and $n_b(i+1)$ such that $k \otimes_A \Ker(\delta_i) \simeq \bar{J}_{n_a(i+1)}^a\oplus \bar{J}_{n_b(i+1)}^b$ if and only if 
the tuples $(E^{s},F^{s},G^{s}_{a},G^{s}_{b})$ of subspaces defined in this subsection are distributive for each $s \in \{ a, b \}$ and for all $n \in \NN_{0}$ (or just $n \geq n_{s}(i)$), 
and we have the following inclusions 
\[     E^{a} \cap \big(X^{a,n-n_{a}(i)}_{b} \otimes V^{(n_{a}(i-2))}\big) 
        \subseteq V^{(n - n_{a}(i+1))} \otimes \bar{J}_{n_a(i+1)}^a + I_{n-n_{a}(i)} \otimes \bar{J}_{n_a(i)}^a,
\]
for all $n \in \NN_{0}$ satisfying that $n_{a}(i-1) + b \leq n \leq n_{a}(i)+b-1$ if $i$ is even, and the inclusion 
\begin{multline*}
      (V^{(b-1)} \otimes \bar{J}_{n_{a}(i)}^a) \cap \big((V^{(s'-1)} \otimes R_{s}) \cap (R_{s'} \otimes V^{(s-1)})\big) \otimes V^{(n_{s}(i-2))}
       \\ \subseteq V^{(b-a)} \otimes \bar{J}_{n_a(i+1)}^a + I_{b-1} \otimes \bar{J}_{n_a(i)}^a,
\end{multline*}
if $i$ is odd. 
\end{proposition}
%%%%%%%
\begin{proof}
We shall first prove the ``if'' part, so we assume the distributivity of the mentioned tuples, and the previous inclusions.    

We have already seen that 
\[
(\Ker(\delta_i))_n 
= \bigoplus_{s=a,b} \frac{(V^{(n-n_{s}(i))} \otimes \bar{J}_{n_{s}(i)}^s) \cap (I_{n-n_{s}(i-1)}\otimes V^{(n_{s}(i-1))})}{I_{n-n_{s}(i)}\otimes \bar{J}_{n_{s}(i)}^s}.
\]
If we denote the numerator of the summand indexed by $s$ on the right member by $N_{n}^{s}$, we have that 
\begin{align*}
N_n^{s} &= E^{s} \cap (F^{s}+G^{s}_{a}+G^{s}_{b})
= (E^{s} \cap F^{s})+(E^{s} \cap G^{s}_{a})+(E^{s} \cap G^{s}_{b})
\\
&= I_{n-n_{s}(i)} \otimes \bar{J}_{n_{s}(i)}^s + V^{(n-n_{s}(i+1))} \otimes \bar{J}_{n_{s}(i+1)}^s,
\end{align*}
where we have used the distributivity conditions in the second equality, equations \eqref{eq:prigen}, \eqref{eq:seggen}, 
either \eqref{eq:tergen} if $i$ is even, or \eqref{eq:tergeniodd} if $i$ is odd, and \eqref{eq:tergenb}, and the inclusions of the statement. 
Hence, $\Ker(\delta_i)$ is $2$-pure in degrees $n_{a}(i+1)$ and $n_{b}(i+1)$ satisfying in fact that $k \otimes_A \Ker(\delta_i) \simeq \bar{J}_{n_{a}(i+1)}^a\oplus \bar{J}_{n_{b}(i+1)}^b$.

Conversely, assume that $\Ker(\delta_i)$ is $2$-pure in degrees $n_{a}(i+1)$ and $n_{b}(i+1)$ such that $k \otimes_A \Ker(\delta_i) \simeq \bar{J}_{n_{a}(i+1)}^a\oplus \bar{J}_{n_{b}(i+1)}^b$.
We will prove the required distributivity condition and the stated inclusions.   
Consider the following commutative diagram for $s \in \{ a, b \}$
\[
\xymatrix@C-80pt
{
\displaystyle{\frac{V^{(n-n_{s}(i+1))}\otimes \bar{J}_{n_{s}(i+1)}^s}{I_{n-n_{s}(i+1)}\otimes \bar{J}_{n_{s}(i+1)}^s}}
\ar@{->>}[rr]
\ar@{->>}[rdd]^{g}
&
& 
\displaystyle{\frac{V^{(n-n_{s}(i+1))}\otimes \bar{J}_{n_{s}(i+1)}^s+I_{n-n_{s}(i)}\otimes \bar{J}_{n_{s}(i)}^s}{I_{n-n_{s}(i)}\otimes \bar{J}_{n_{s}(i)}^s}}
\ar@{_{(}->}[ldd]^{f}
\\
&
& 
\\ 
&
\displaystyle{\frac{(V^{(n-s-1)}\otimes \bar{J}_{s+1}^s)\cap (I_{n-s}\otimes V^{(s)})}{I_{n-s-1}\otimes \bar{J}_{s+1}^s}}
& 
}
\]
Notice that the domain of $g$ is the homogeneous component of degree $n$ of $A \otimes \bar{J}_{n_{s}(i+1)}^{s}$, its codomain is 
the component of degree $n$ of $\Ker(\delta_{i}^{s})$ and $g$ is the corresponding homogeneous component of the projective cover, so surjective. 
The horizontal epimorphism follows from the trivial inclusion 
\[     I_{n-n_{s}(i+1)}\otimes \bar{J}_{n_{s}(i+1)}^s \subseteq (V^{(n-n_{s}(i+1))}\otimes \bar{J}_{n_{s}(i+1)}^s)\cap (I_{n-n_{s}(i)}\otimes \bar{J}_{n_{s}(i)}^s),     \] 
and the map $f$ is just the canonical inclusion. 
The commutativity of the diagram just follows from the construction of the projective cover of a module. 
The surjectivity of $g$ yields the surjectivity of $f$, so the latter is an isomorphism. 
If we consider $i$ even and the case $n_{a}(i-1) + b \leq n \leq n_{a}(i)+b-1$, the equality coming from \eqref{eq:tergen} and the identity given by $f$ imply that the inclusions for the even case hold. 
On the other hand, if $i$ is odd and we take $n = n_{a}(i)+b-1$, the equality \eqref{eq:tergeniodd} tells us that the inclusion in the odd case also holds. 
Finally, the distributivity of the corresponding tuple also follows from the isomorphism $f$. 
The proposition is thus proved. 
\end{proof}

%%%%%%%
\begin{remark}
\label{rem:isoindi}
In analogous manner to what was stated in Remark \ref{rem:isoind2}, the condition $k \otimes_{A} \Ker(\delta_{i}) \simeq \bar{J}_{n_{a}(i+1)}^{a} \oplus \bar{J}_{n_{b}(i+1)}^{b}$, for $i \geq 3$, 
usually understood as stating that the obvious morphism from $\bar{J}_{n_{a}(i+1)}^{a} \oplus \bar{J}_{n_{b}(i+1)}^{b}$ to $k \otimes_{A} \Ker(\delta_{i})$ is an isomorphism, 
is equivalent to the weaker condition given by the existence of any isomorphism $k \otimes_{A} \Ker(\delta_{i}) \simeq \bar{J}_{n_{a}(i+1)}^{a} \oplus \bar{J}_{n_{b}(i+1)}^{b}$, 
due to Lemma \ref{lemma:tresimp} (see also the proof of Proposition \ref{prop:tresimp}). 
\end{remark}
%%%%%%%

%%%%%%%%%%%%%%%%%%%%%%%%%%%%%%%%%%%%%%%%%%%%%%%%%%%%%%%%%%%%%%%%%%%%%%%%%%%%%%%%%%%%%%%%%%%%%%%%%%%%%%%%%%%%%%%%%%%%%%%%%%%%%%%%%%%%%%%%%%%%%%%%%%%%%%%%%%%%

\subsection{\texorpdfstring{Main result}{Main result}}

We shall summarize in this subsection the main result achieved in the description  of the $\{a,b\}$-multi-Koszul property of algebras in terms of lattices of subspaces, which follows from the previous results proved in the two previous subsections.

%%%%%%%
\begin{theorem}
\label{thm:Kozulequiv}
Let $A=T(V)/\cl{R}$ be an $\{a,b\}$-multi-homogeneous algebra ($2 \leq a<b$) such that $R_a \oplus R_b$ satisfies the minimality condition. 
We have thus the following equivalences:
\begin{itemize}
\item[(i)] The \textit{e.c.} are satisfied, we have the following collection of inclusions 
\[     E^{a} \cap \big(X^{a,n-n_{a}(i)}_{b} \otimes V^{(n_{a}(i-2))}\big) 
        \subseteq V^{(n - n_{a}(i+1))} \otimes \bar{J}_{n_a(i+1)}^a + I_{n-n_{a}(i)} \otimes \bar{J}_{n_a(i)}^a,
\]
for all $n \in \NN_{0}$ satisfying that $n_{a}(i-1) + b \leq n \leq n_{a}(i)+b-1$ if $i$ is even, and the inclusion 
\begin{multline*}
      (V^{(b-1)} \otimes \bar{J}_{n_{a}(i)}^a) \cap \big((V^{(s'-1)} \otimes R_{s}) \cap (R_{s'} \otimes V^{(s-1)})\big) \otimes V^{(n_{s}(i-2))}
       \\ \subseteq V^{(b-a)} \otimes \bar{J}_{n_a(i+1)}^a + I_{b-1} \otimes \bar{J}_{n_a(i)}^a,
\end{multline*}
if $i$ is odd, and the tuple $(E',E'',F',G',F'',G'')$ considered in Proposition \ref{prop:kerdelta2pure} is $(2,2)$-bidistributive, 
and for all $i \geq 3$ the tuples $(E^{s},F^{s},G^{s}_{a},G^{s}_{b})$ of subspaces defined in the previous Subsection 
are distributive for each $s \in \{ a, b \}$ and for all $n \in \NN_{0}$ (or just $n \geq n_{s}(i)$). 
\item[(ii)] $A$ is $\{a,b\}$-multi-Koszul.
\end{itemize}
\end{theorem}
%%%%%%%
\begin{proof}
The first statement is a consequence of the second one by Propositions \ref{prop:kerdelta2pure} and \ref{prop:kerdeltai}. 
The converse is also direct, using the mentioned results and taking into account the recursive process explained in the first paragraph of the previous subsection. 
\end{proof}

%%%%%%%%%%%%%%%%%%%%%%%%%%%%%%%%%%%%%%%%%%%%%%%%%%%%%%%%%%%%%%%%%%%%%%%%%%%%%%%%%%%%%%%%%%%%%%%%%%%%%%%%%%%%%%%%%%%%%%%%%%%%%%%%%%%%%%%%%%%%%%%%%%%%%%%%%%%%

\subsection{\texorpdfstring{The monomial case}{The monomial case}}

We shall consider in this subsection an algebra $A=T(V)/\cl{R}$ such that $R = R_a \oplus R_b$ ($2 \leq a<b$) satisfies the minimality condition 
and it has a basis of monomials. 

We shall first recall some definitions.  
The set of all the subspaces of $V$ will be denoted by $\mathcal{L}(V)$. 
It is known that the lattice $(\mathcal{L}(V), \subseteq, +, \cap)$ is modular, \textit{i.e.} given $W_1,W_2,W_3 \in \mathcal{L}(V)$, 
if $W_2 \subseteq W_1$, then $W_1 \cap (W_2 + W_3) = W_2 + (W_1\cap W_3)$. 
A sublattice $\mathcal{S}\subseteq \mathcal{L}(V)$ is \emph{distributive} if $E \cap (F+G)=(E\cap F)+(E\cap G)$ for all $E,F,G \in \mathcal{S}$. 
Note that this also implies the condition $E+(F\cap G)=(E+F)\cap (E+G)$ (see \cite{O}, Ch. I, \S 4, Thm. 3).

The following result is a first criterion for distributivity. 
%%%%%%%
\begin{proposition}
\label{prop:distbasis}
Given $W_1,\dots ,W_n$ subspaces of $V$, we consider the sublattice $\mathcal{T}$ generated by $W_1,\cdots,W_n$, 
\textit{i.e.} $\mathcal{T}$ is the intersection of all the sublattices of $\mathcal{L}(V)$ containing the subspaces $W_1,\cdots ,W_n$.
Then $\mathcal{T}$ is distributive if and only if there exists a basis $\mathcal{B}$ of $V$ such that $\mathcal{B}_i = \mathcal{B}\cap W_i$ is a basis of $W_i$ for all $1\le i\le n$. 
In this case, we say that $\mathcal{B}$ distributes with respect to $W_1,\cdots ,W_n$.
\end{proposition}
%%%%%%%
\begin{proof}
See \cite{Ba}, Lemma 1.2. 
\end{proof}

We may apply the previous result to the situation we are interested in: 
if $R$ has a basis of monomials, the sublattice of all vector subspaces of $V^{(n)}$ generated by $V^{(j)} \otimes R_{s} \otimes V^{(n-s-j)}$, for $j = 0, \dots, n - s$ and $s = a, b$, is distributive, 
as one can deduce by using the basis of $V^{(n)}$ composed of all monomials of (tensor) degree $n$.
This implies that all the tuples considered in Propositions \ref{prop:kerdelta2pure} and \ref{prop:kerdeltai} are distributive. 
In fact, using Lemma \ref{lem:crossedincl} and the comments of Remark \ref{rem:incimp}, we get that $\Ker(\delta_2)$ is $2$-pure in degrees $a+1$ and $b+1$ such that 
$k \otimes_A \Ker(\delta_2) \simeq \bar{J}_{a+1}^a\oplus \bar{J}_{b+1}^b$ if and only if the \textit{e.c.} and the first collection of inclusions of Remark \ref{rem:incimp} hold. 
Moreover, we remark that in this case Proposition \ref{prop:kerdeltai} can be rewritten as stating that, under the same assumptions, 
$\Ker(\delta_i)$ is $2$-pure in degrees $n_a(i+1)$ and $n_b(i+1)$ such that $k \otimes_A \Ker(\delta_i) \simeq \bar{J}_{n_a(i+1)}^a\oplus \bar{J}_{n_b(i+1)}^b$ if and only if 
the tuples $(E^{s},F^{s},G^{s}_{a},G^{s}_{b})$ of subspaces defined in that subsection are distributive for each $s \in \{ a, b \}$ and for all $n \in \NN_{0}$ (or just $n \geq n_{s}(i)$). 
This is due to the fact that the inclusion at the end of the proposition follows from inclusion \eqref{eq:tergen} and the mentioned distributivity. 
We have thus proved the following result: 

%%%%%%%
\begin{corollary}
\label{coro:monomial}
Let $A=T(V)/\cl{R}$ be an $\{a,b\}$-multi-homogeneous algebra ($2 \leq a<b$) such that $R_a \oplus R_b$ satisfies the minimality condition 
and it has a basis of monomials. 
Then, $A$ is multi-Koszul if and only if the \textit{e.c.} and the inclusions of Remark \ref{rem:incimp} are satisfied. 
\end{corollary}
%%%%%%%

\begin{remark}
Note that if the algebra $A$ is monomial the \textit{e.c.} may be equivalently stated as an overlapping property on a basis of monomials of the space of relations, as in \cite{B2}, Prop. 3.8. 
\end{remark}
%%%%%%%

%%%%%%%%%%%%%%%%%%%%%%%%%%%%%%%%%%%%%%%%%%%%%%%%%%%%%%%%%%%%%%%%%%%%%%%%%%%%%%%%%%%%%%%%%%%%%%%%%%%%%%%%%%%%%%

\section*{\texorpdfstring{Acknowledgements}{sec:acknowledgements}}
\addcontentsline{toc}{section}{Acknowledgements}

The second author would like to thank Eduardo Marcos and Andrea Solotar for interesting comments and suggestions.

%%%%%%%%%%%%%%%%%%%%%%%%%%%%%%%%%%%%%%%%%%%%%%%%%%%%%%%%%%%%%%%%%%%%%%%%%%%%%%%%%%%%%%%%%%%%%%%%%%%%%%%%%%%%%%%%%%%%%%%%%%%%%%%%%%%%%%%%%%%%%%%%%%%%%

\bibliographystyle{model1-num-names}

\begin{bibdiv}
\begin{biblist}

\bib{Ba}{thesis}{
   author={Backelin, J{\"o}rgen},
   title={A distributiveness property of augmented algebras and some related homological results},
   type={Ph.D. Thesis},
   place={Stockholm},
   date={1982},
}

\bib{BF}{article}{
   author={Backelin, J{\"o}rgen},
   author={Fr{\"o}berg, Ralf},
   title={Koszul algebras, Veronese subrings and rings with linear
   resolutions},
   journal={Rev. Roumaine Math. Pures Appl.},
   volume={30},
   date={1985},
   number={2},
   pages={85--97},
}

\bib{BGS1}{article}{
   author={Be{\u\i}linson, A. A.},
   author={Ginsburg, V. A.},
   author={Schechtman, V. V.},
   title={Koszul duality},
   journal={J. Geom. Phys.},
   volume={5},
   date={1988},
   number={3},
   pages={317--350},
}

\bib{BGS2}{article}{
   author={Beilinson, Alexander},
   author={Ginzburg, Victor},
   author={Soergel, Wolfgang},
   title={Koszul duality patterns in representation theory},
   journal={J. Amer. Math. Soc.},
   volume={9},
   date={1996},
   number={2},
   pages={473--527},
}

\bib{B2}{article}{
   author={Berger, Roland},
   title={Koszulity for nonquadratic algebras},
   journal={J. Algebra},
   volume={239},
   date={2001},
   number={2},
   pages={705--734},
   note={See also \textit{Koszulity for nonquadratic algebras II}. Preprint available at \texttt{arXiv:math/0301172V1 [math.QA]}},
}

\bib{B4}{article}{ 
   author={Berger, Roland}, 
   title={La cat\'egorie des modules gradu\'es sur une alg\`ebre gradu\'ee (nouvelle version du chapitre 5 d'un cours de Master 2 \'a Lyon 1)},
   date={2008},
   url={http://webperso.univ-st-etienne.fr/$\sim$rberger/mes-textes.html},
}

\bib{BDW}{article}{
   author={Berger, Roland},
   author={Dubois-Violette, Michel},
   author={Wambst, Marc},
   title={Homogeneous algebras},
   journal={J. Algebra},
   volume={261},
   date={2003},
   number={1},
   pages={172--185},
}

\bib{BG}{article}{
   author={Berger, Roland},
   author={Ginzburg, Victor},
   title={Higher symplectic reflection algebras and non-homogeneous
   $N$-Koszul property},
   journal={J. Algebra},
   volume={304},
   date={2006},
   number={1},
   pages={577--601},
}

\bib{BM}{article}{
   author={Berger, Roland},
   author={Marconnet, Nicolas},
   title={Koszul and Gorenstein properties for homogeneous algebras},
   journal={Algebr. Represent. Theory},
   volume={9},
   date={2006},
   number={1},
   pages={67--97},
}

\bib{BBK}{article}{
   author={Brenner, Sheila},
   author={Butler, Michael C. R.},
   author={King, Alastair D.},
   title={Periodic algebras which are almost Koszul},
   journal={Algebr. Represent. Theory},
   volume={5},
   date={2002},
   number={4},
   pages={331--367},
}

\bib{C}{book}{
   title={S\'eminaire Henri Cartan, 11e ann\'e: 1958/59. Invariant de Hopf
   et op\'erations cohomologiques secondaires},
   language={French},
   series={2e \'ed. 2 vols. \'Ecole Normale Sup\'erieure},
   publisher={Secr\'etariat math\'ematique},
   place={Paris},
   date={1959},
   pages={Vol. 1 (exp. 1--9), ii+121 pp. Vol. 2 (exp. 10--19), ii+158 pp.
   (mimeographed)},
}

\bib{CG}{article}{
   author={Conner, Andrew},
   author={Goetz, Pete},
   title={$A_\infty$-algebra structures associated to $\scr K_2$
   algebras},
   journal={J. Algebra},
   volume={337},
   date={2011},
   pages={63--81},
}

\bib{CS}{article}{
   author={Cassidy, Thomas},
   author={Shelton, Brad},
   title={Generalizing the notion of Koszul algebra},
   journal={Math. Z.},
   volume={260},
   date={2008},
   number={1},
   pages={93--114},
}

\bib{F}{article}{
   author={Fr{\"o}berg, R.},
   title={Koszul algebras},
   conference={
      title={Advances in commutative ring theory},
      address={Fez},
      date={1997},
   },
   book={
      series={Lecture Notes in Pure and Appl. Math.},
      volume={205},
      publisher={Dekker},
      place={New York},
   },
   date={1999},
   pages={337--350},
}

\bib{Go}{article}{
   author={Govorov, V. E.},
   title={Dimension and multiplicity of graded algebras},
   language={Russian},
   journal={Sibirsk. Mat. \v Z.},
   volume={14},
   date={1973},
   pages={1200--1206, 1365},
}

\bib{GM}{article}{
   author={Green, Edward L.},
   author={Marcos, E. N.},
   title={$d$-Koszul algebras, 2-$d$-determined algebras and 2-$d$-Koszul
   algebras},
   journal={J. Pure Appl. Algebra},
   volume={215},
   date={2011},
   number={4},
   pages={439--449},
}

\bib{GMMZ}{article}{
   author={Green, E.},
   author={Marcos, E},
   author={Mart\'{i}nez-Villa, R.},
   author={Zhang, P.},
   title={$D$-Koszul algebras},
   journal={J. Pure and App. Algebra},
   volume={193},
   date={2004},
   pages={141--162},
}

\bib{HL}{article}{
   author={Hai, Ph{\`u}ng H{\^o}},
   author={Lorenz, Martin},
   title={Koszul algebras and the quantum MacMahon master theorem},
   journal={Bull. Lond. Math. Soc.},
   volume={39},
   date={2007},
   number={4},
   pages={667--676},
}

\bib{HeL}{article}{
   author={He, Ji-Wei},
   author={Lu, Di-Ming},
   title={Higher Koszul algebras and $A$-infinity algebras},
   journal={J. Algebra},
   volume={293},
   date={2005},
   number={2},
   pages={335--362},
}

\bib{K}{article}{
   author={Koszul, Jean-Louis},
   title={Homologie et cohomologie des alg\`ebres de Lie},
   language={French},
   journal={Bull. Soc. Math. France},
   volume={78},
   date={1950},
   pages={65--127},
}

\bib{LH}{thesis}{
   author={Lef\`evre-Hasegawa, Kenji},
   title={sur les $A_{\infty}$-cat\'egories},
   language={French},
   type={Ph.D. Thesis},
   place={Paris},
   date={2003},
   note={Corrections at \texttt{http://www.math.jussieu.fr/~keller/lefevre/TheseFinale/corrainf.pdf}},
}

\bib{M}{article}{
   author={Manin, Yu. I.},
   title={Some remarks on Koszul algebras and quantum groups},
   language={English, with French summary},
   journal={Ann. Inst. Fourier (Grenoble)},
   volume={37},
   date={1987},
   number={4},
   pages={191--205},
}

\bib{Mer}{article}{
   author={Merkulov, S. A.},
   title={Strong homotopy algebras of a K\"ahler manifold},
   journal={Internat. Math. Res. Notices},
   date={1999},
   number={3},
   pages={153--164},
}

\bib{NV}{book}{
   author={N{\u{a}}st{\u{a}}sescu, Constantin},
   author={Van Oystaeyen, Freddy},
   title={Methods of graded rings},
   series={Lecture Notes in Mathematics},
   volume={1836},
   publisher={Springer-Verlag},
   place={Berlin},
   date={2004},
   pages={xiv+304},
}

\bib{O}{article}{
   author={Ore, Oystein},
   title={On the foundation of abstract algebra. I},
   journal={Ann. of Math. (2)},
   volume={36},
   date={1935},
   number={2},
   pages={406--437},
}

\bib{P}{article}{
   author={Priddy, Stewart B.},
   title={Koszul resolutions},
   journal={Trans. Amer. Math. Soc.},
   volume={152},
   date={1970},
   pages={39--60},
}

\bib{W}{book}{
   author={Weibel, Charles A.},
   title={An introduction to homological algebra},
   series={Cambridge Studies in Advanced Mathematics},
   volume={38},
   publisher={Cambridge University Press},
   place={Cambridge},
   date={1994},
   pages={xiv+450},
}

\end{biblist}
\end{bibdiv}
\addcontentsline{toc}{section}{References}

%%%%%%%%%%%%%%%%%%%%%%%%%%%%%%%%%%%%%%%%%%%%%%%%%%%%%%%%%%%%%%%%%%%%%%%%%%%%%%%%%%%%%%%%%%%%%%%%%%%%%%%%%%%%%%%%%%%%%%%%%%%%%%%%%%%%%%%%%%%%%%%%%%%%%%%%%

%%%%%%%%%%%

\end{document}